\documentclass{amsart}
\pdfoutput=1
\usepackage{amsfonts}
\usepackage{amsmath}
\usepackage{amssymb}
\usepackage{eucal}
\usepackage{graphicx}
\usepackage{color}
\definecolor{darkgreen}{rgb}{0.0,0.5,0.0}
\definecolor{navy}{rgb}{0.0,0.0,0.5}
\definecolor{ocher}{rgb}{0.6,0.4,0.0}
\definecolor{sepia}{rgb}{0.4,0.2,0.0}
\usepackage{hyperref}
\hypersetup{colorlinks,
            citecolor=blue,
            filecolor=red,
            linkcolor=ocher,
            urlcolor=darkgreen,
            bookmarksopen=true, pdftex}

\newtheorem{theorem}{Theorem}
\theoremstyle{plain}
\newtheorem{lemma}{Lemma}

\numberwithin{equation}{section}
\newcommand{\N}{\mathbb{N}}
\newcommand{\R}{\mathbb{R}}
\newcommand{\br}{\bar{r}}
\newcommand{\bv}{\bar{v}}
\newcommand{\bt}{\bar{t}}
\newcommand{\cD}{\mathcal{D}}
\newcommand{\cF}{\mathcal{F}}
\newcommand{\cK}{\mathcal{K}}
\newcommand{\cL}{\mathcal{L}}
\newcommand{\cM}{\mathcal{M}}
\newcommand{\cN}{\mathcal{N}}
\newcommand{\cO}{\mathcal{O}}
\newcommand{\cQ}{\mathcal{Q}}
\newcommand{\cS}{\mathcal{S}}
\newcommand{\tg}{\tilde{g}}
\newcommand{\tpsi}{\tilde{\psi}}
\newcommand{\tv}{\tilde{v}}
\newcommand{\gcan}{g_{\mathrm{can}}}
\newcommand{\Rc}{\mathrm{Rc}}
\newcommand{\Rm}{\mathrm{Rm}}
\newcommand{\OUT}{{\mathrm{out}}}
\newcommand{\IN}{{\mathrm{in}}}
\newcommand{\PAR}{{\mathrm{par}}}
\newcommand{\COL}{\mathrm{col}}
\newcommand{\NP}{\mathrm{NP}}
\newcommand{\SP}{\mathrm{SP}}
\newcommand{\init}{\mathrm{init}}
\newcommand{\bry}{\mathfrak{B}}
\newcommand{\cry}{\mathfrak{C}}
\newcommand{\vini}{v_\init}
\newcommand{\uini}{u_\init}
\newcommand{\amax}{\mathcal{A}}

\begin{document}

\title[Minimally invasive surgery for Ricci flow]
{Minimally invasive surgery\\for Ricci flow singularities}

\author{Sigurd B.~Angenent} \address[Sigurd Angenent]{ University of
Wisconsin-Madison} \email{angenent@math.wisc.edu}
\urladdr{http://www.math.wisc.edu/\symbol{126}angenent/}

\author{M.~Cristina Caputo} \address[M.~Cristina Caputo]{ University of
Texas at Austin} \email{caputo@math.utexas.edu}
\urladdr{http://www.ma.utexas.edu/users/caputo/}

\author{Dan Knopf} \address[Dan Knopf]{ University of Texas at Austin}
\email{danknopf@math.utexas.edu}
\urladdr{http://www.ma.utexas.edu/users/danknopf/}

\thanks{S.B.A.~acknowledges NSF support via DMS-0705431.  M.C.C.~and
D.K.~acknowledge NSF support via DMS-0545984. M.C.C.~ thanks the
Max Planck Institut f\"ur Gravitationsphysik (AEI), whose hospitality
she enjoyed during part of the time this paper was in preparation.}

\begin{abstract}
In this paper, we construct smooth forward Ricci flow evolutions of
singular initial metrics resulting from rotationally symmetric neckpinches on
$\cS^{n+1}$, without performing an intervening surgery. In the restrictive
context of rotational symmetry, this construction gives evidence in favor
of Perelman's hope for a ``canonically defined Ricci flow through singularities''.
\end{abstract}

\maketitle
\tableofcontents

\section{Introduction}

Finite time singularity formation is in a sense a ``generic'' property of
Ricci flow.  For example, on any Riemannian manifold where the maximum
principle applies, a solution $(\cM^{n},g(t))$ whose scalar curvature at
time $t=0$ is bounded from below by a positive constant $r$ must become
singular at or before the formal vanishing time $T_{\mathrm{form}}=n/2r$.
In some cases, e.g.~if the curvature operator of the initial metric is
sufficiently close to that of a round sphere, the entire manifold
disappears in a global singularity.  A beautiful example is the result of
Brendle and Schoen that a compact manifold with pointwise $1/4$-pinched
sectional curvatures will shrink to a round point \cite{BS08}.  Under more
general conditions, one expects formation of a local singularity at
$T<\infty$.  Here, there exists an open set $\Omega$ of $\cM$ such that
$\limsup_{t\nearrow T}\left| \Rm (x,t)\right| <\infty$ for points
$x\in\Omega$.

Such behavior has long been strongly conjectured; see e.g.~Hamilton's
heuristic arguments \cite[Section~3] {Ham95}.  All known rigorous
Riemannian (i.e.~non-K\"{a}hler) examples involve ``necks'' forming under
certain symmetry hypotheses.  Local singularity formation was first
established by Simon on noncompact manifolds \cite{Simon00}.  Neckpinch
singularities for metrics on $\cS^{n+1}$ were studied by two of the authors
\cite{AK04, AK07}.  Gu and Zhu established existence of local Type-II
(i.e.~slowly forming) ``degenerate neckpinch'' singularities on
$\cS^{n+1}$ \cite{GZ08}.

Continuing a solution of Ricci flow past a singular time $T<\infty$ has
always been done by surgery.  Hamilton introduced a surgery algorithm for
compact $4$-manifolds of positive isotropic curvature \cite{Ham97}.
(Huisken and Sinestrari have a related surgery program for solutions of
mean curvature flow \cite{HS09}.) Perelman developed a somewhat different
surgery algorithm for compact $3$-manifolds \cite{P1, P2, P3}.  A solution
of Ricci Flow with Surgery (\textsc{rfs}) is a sequence
$(\cM_{k}^{n},g(t):T_{k}^{-}\leq t<T_{k}^{+})$ of maximal smooth solutions
of Ricci flow such that at each discrete surgery time
$T_{k}^{+}=T_{k+1}^{-}$, the smooth manifold $\bigl( \cM_{k+1}^3
,g(T_{k+1}^{-})\bigr) $ is obtained from the singular limit $\bigl(
\cM_{k}^3,g(T_{k}^{+})\bigr) = \bigl( \cM_{k}^3,\lim _{t\nearrow
T_{k}^{+}}g(t)\bigr)$ by a known topological-geometric modification.
Each geometric modification depends on a number of choices (e.g.~surgery scale,
conformal factors and cut-off functions) that must be made carefully so
that critical \emph{a priori} estimates are preserved.  The technical
details of \textsc{rfs} for a $3$-manifold are discussed extensively in the
literature; see e.g.~\cite{KL07}, \cite{CZ06}, \cite{MT07}, and
\cite{Tao06}.

It is tempting to ask whether the choices made in surgery can somehow be
eliminated.  Indeed, Perelman conjectures \cite[Section~13.2]{P1} that the
following important and natural question has an affirmative answer:

\begin{quote}
  \textbf{Question 1 } (Perelman).  Let $g_{i}(t)$ denote a smooth Ricci
  flow solution obtained by performing surgery on singular data
  $(\cM^3,g(T))$ at a scale $h_{i}>0$.  Does the sequence $\{ \cM^3,
  g_{i}(t):T\leq t<T_{i}\} $ have a well-defined limit as $h_{i}\searrow0$\,?
\end{quote}

He writes: ``It is likely that by passing to the limit in this construction
one would get a canonically defined Ricci flow through singularities, but
at the moment I don't have a proof of that.''  If one regards a sequence
$\{ (\cM ^{n},g_{i}(T)):i\in\N\}$ of surgically modified initial data as a
sequence of smooth approximations to irregular initial data
$(\cM^{n},g(T))$, Question~1 may be regarded as a problem of showing that
the Ricci flow system of \textsc{pde} is well posed with respect to a
particular regularization scheme.

Somewhat more generally, Perelman's wish for a ``canonically defined Ricci
flow through singularities'' might be rephrased as follows:

\begin{quote}
  \textbf{Question 2 } (Perelman).  Is it possible to flow directly out of
  a local singularity without arbitrary choices?
\end{quote}

In this paper, we consider $\textrm{SO}(n+1, \R)$-invariant metrics on
$\cS^{n+1}$, and within this restricted context provide positive answers to
both Questions~1 and 2 by exhibiting forward evolutions of the rotationally
symmetric neckpinch.  We say a smooth complete solution $(\cM^{n}
,g(t):T<t<T')$ of Ricci flow is a \emph{forward evolution} of a singular
metric $(\cM^{n},\hat{g}(T))$ if, as $t\searrow T$, the metric
$g(t)|_{\cO}$ converges smoothly to $\hat{g}(T)|_{\cO}$ on any open subset
$\cO\subset\cM$ for which $\hat{g}$ is regular on $\overline\cO$.  Thus we
effectively let the Ricci flow \textsc{pde} perform surgery at scale zero.
In so doing, we show that any forward evolution with the same symmetries as
$(\cM^{n},\hat{g}(T))$ must have a precise asymptotic profile as it emerges
from the singularity.  Although the hypothesis of rotational symmetry is
highly restrictive,\footnote{On the other hand, formal matched asymptotics
for fully general neckpinches predict that every neckpinch is
asymptotically rotationally symmetric \cite[Section~3]{AK07}.} our
construction provides examples of what a canonical flow through
singularities would look like if Perelman's hope could be answered
affirmatively in general.

There exist a few examples in the literature of non-smooth initial data
evolving by Ricci flow, none of which apply to the situation considered
here.  Bemelmans, Min-Oo, and Ruh applied Ricci flow to regularize $C^2$
initial metrics with bounded sectional curvatures \cite{BMR84}.  Simon used
Ricci flow modified by diffeomorphisms to evolve complete $C^0$ initial
metrics that can be uniformly approximated by smooth metrics with bounded
sectional curvatures \cite{Simon02}.  Simon also evolved $3$-dimensional
metric spaces that arise as Gromov--Hausdorff limits of sequences of
complete Riemannian manifolds of almost nonnegative curvatures whose
diameters are bounded away from infinity and whose volumes are bounded away
from zero \cite{Simon08}.  Koch and Lamm demonstrated global existence and
uniqueness for Ricci--DeTurck flow of initial data that are close to the
Euclidean metric in $L^{\infty }(\R^{n})$ \cite{KL09}; this may be regarded
as a generalization of a stability result of Schn\"{u}rer, Schulze, and
Simon \cite{SSS07}.  In recent work, Chen, Tian, and Zhang defined and
studied \textit{weak solutions} of K\"{a}hler--Ricci flow whose initial
data are K\"{a}hler currents with bounded $C^{1,1}$ potentials
\cite{CTZ08}.  The special case of conformal initial data $e^{u}g$ on a
compact Riemannian surface $(\cM^2,g)$ with $e^{u}\in L^{\infty}(\cM^2)$
was considered earlier by Chen and Ding \cite{CD07}.  (Though their proofs
are quite different, both of these papers take advantage of circumstances
in which Ricci flow reduces to a scalar evolution equation.) The term
``weak solutions'' was used in a different context by Bessi\`{e}res,
Besson, Boileau, Maillot, and Porti to describe certain solutions of
\textsc{rfs} on compact, irreducible non-spherical $3$-manifolds
\cite{BBBMP08} .  In another direction, Topping analyzed Ricci flow of
incomplete initial metrics on surfaces with Gaussian curvature bounded from
above \cite{Topping09}.

\bigskip

This paper and its results are organized as follows.  In
Section~\ref{sec:Setup}, we use rotational symmetry to simplify our
problem. The assumption of rotational symmetry generally allows one to reduce
the full Ricci flow system to a scalar parabolic \textsc{pde} in one space
dimension. For this problem, we were not able to find a convenient global
description of the solution in terms of a one-dimensional scalar heat
equation. However, in Section~\ref{sec:Setup}, we show that, at least in
appropriate \emph{local} coordinates valid in a neighborhood of the
singular point, a forward evolution of Ricci flow
\begin{equation}
  \frac{\partial g}{\partial t}=-2\Rc \label{ForwardEvolution}
\end{equation}
emerging from a rotationally symmetric neckpinch singularity at time zero
is equivalent to a smooth positive solution of the quasilinear \textsc{pde}
\begin{equation}
  v_{t}=vv_{rr}-\tfrac12v_{r}^2+\frac{n-1-v}{r}v_{r}+\frac{2(n-1)}{r^2
  }\left(  v-v^2\right)
\end{equation}
emerging from singular initial data which satisfy
\begin{equation}
  \vini(r)=[1+o(1)]v_0(r)\qquad\text{as }r\searrow0,
  \label{AndSoItAllBegins}
\end{equation}
where
\begin{equation}
  v_0(r) \doteqdot\frac{\frac14(n-1)}{-\log r}.  \label{eq:define-v-tilde}
\end{equation}
Note that a smooth \emph{forward evolution} of (\ref{AndSoItAllBegins})
must satisfy $\lim_{t\searrow0}v(r,t)=\vini(r)$ at all points where the
initial metric is nonsingular, i.e.~at all $r>0$.

There are only two ways that the solution (\ref{ForwardEvolution}) can be
complete.  The first is that $v$ satisfies the smooth boundary condition
$v(0,t)=1$ for all $t>0$ that it exists.  Because $\lim_{r\searrow0}\vini(r)=0$,
this is incompatible with the initial data, meaning that $v$ immediately jumps
at the singular hypersurface $\{0\}\times\cS^{n}$, yielding a compact
forward evolution that heals the singularity with a smooth $n$-ball.  In
this case, the sectional curvatures immediately become bounded in space, at
least for a short time.  The second possibility is that $v$ remains
singular at $r=0$, but the distance to the singularity measured with
respect to $g(t)$ immediately becomes infinite, yielding a noncompact
forward evolution, necessarily with unbounded sectional curvatures.  In
Section~\ref{Compact}, we show that the second possibility cannot occur: we
prove that any smooth complete rotationally symmetric forward Ricci flow
evolution from a rotationally symmetric neckpinch is compact.  (See
Theorem~\ref{CompactResult}.) This result contrasts with Topping's
observation that flat $\R^2$ punctured at a single point can evolve under
Ricci flow by immediately forming a noncompact smooth hyperbolic cusp,
thereby pushing the ``singularity''\ to infinity \cite{Topping09}. In this
section, we also prove that any smooth forward evolution satisfies a unique
asymptotic profile. (See Theorem~\ref{pinch}.)

In Section~\ref{sec:Formal}, we derive formal matched asymptotics for a
solution emerging from a neckpinch singularity.  This discussion is
intended to motivate the rigorous arguments that follow.  Because the
initial $\lim_{r\searrow0}\vini(r)=0$ and boundary $v(0,t)=1$ conditions
are incompatible, one cannot have $v(r,t)\to \vini(r)$ uniformly as
$t\searrow0$.  The solution must resolve this incompatibility in layers.
Consequently, we describe the asymptotic behavior of the formal solution
for small $t>0$ by splitting the $(r,t)$ plane into three regions:
\begin{center}
  \begin{tabular} [c]{ccc}\\
    \hline\vrule width 0pt height 10pt depth 5pt {\bfseries inner} &
    {\bfseries parabolic} & {\bfseries outer}\\ \hline \vrule width 0pt
    height 18pt depth 12pt $r\sim\sqrt{\dfrac{t}{-\log t}}$ & $r\sim \sqrt
    t$
    & $r\sim 1$\\
    \hline
  \end{tabular}
  \medskip
\end{center}

In Section~\ref{sec:Informal}, we make these asymptotics rigorous by
constructing suitable sub- and super- solutions in each space-time region
and ensuring that they overlap properly.

Finally, in Section~\ref{sec:Compactness}, we prove a compactness result that
shows that a subsequence of regularized solutions does converge to a smooth
forward evolution from a neckpinch singularity.
(See Lemmas~\ref{lem:convergence-away-from-NP} and \ref{lem:singularity-off}.)

Here is a slightly glossed summary of our results.  (For more detail,
including how the constants are chosen, see Lemmas~\ref{thm:can-glue} and
\ref{thm:can-glue-too} and Theorems~\ref{thm:outerlayer-subsuper},
\ref{thm:paraboliclayer-subsuper}, and \ref{InnerResult}.)

\begin{theorem}\label{thm:main}
  For $n\geq2$, let $(\cS^{n+1},g_0)$ denote a singular metric arising as
  the limit as $t\nearrow T_0$ of a rotationally symmetric neckpinch
  forming at time $T_0=0$.  Then there exists a complete smooth forward
  evolution $(\cS^{n+1},g(t):T_0<t<T_1)$ of $g_0$ by Ricci flow.

  Any complete smooth forward evolution is compact and satisfies a unique
  asymptotic profile as it emerges from the singularity.  In a local
  coordinate $0<r<r_*\ll1$ such that the singularity occurs at $r=0$ and
  the metric is
  \[
  g(r,t)=\frac{(dr)^2}{v(r,t)}+r^2\gcan ,
  \]
  this asymptotic profile is as follows.

  \textbf{Outer region: }for $c_1\sqrt{t}<r<c_2$, one has
  \[
  v(r,t)=[1+o(1)]\frac{n-1}{-4\log r}\left[ 1+2(n-1)\frac{t}{r^2}\right]
  \qquad\text{uniformly as }t\searrow0.
  \]

  \textbf{Parabolic region: }let $\rho=r/\sqrt{t}$ and $\tau=\log t$; then
  for $c_3/\sqrt{-\tau}<\rho<c_4$, one has
  \[
  v(r,t)=[1+o(1)]\frac{n-1}{-2\tau}\left[ 1+\frac{2(n-1)}{\rho^2}\right]
  \qquad\text{uniformly as }t\searrow0.
  \]

  \textbf{Inner region: }let $\sigma=\sqrt{-\tau}\rho=\sqrt{-\tau/t}\,r$;
  then for $0<\sigma<c_5$, one has
  \[
  v(r,t)=[1+o(1)]\bry (\frac{\sigma}{n-1})\qquad\text{uniformly as
  }t\searrow0,
  \]
  where $\dfrac{(d\sigma)^2}{\bry(\sigma)}+\sigma^2\gcan $ is the Bryant
  soliton metric.
\end{theorem}
We recall the construction and some relevant properties of the Bryant soliton
metric in Appendix~C.
\medskip

Geometrically, the results above admit the following interpretation.  In the inner
layer, going backward in time, one sees a forward evolution emerging from
either side of a neckpinch cusp by forming a Bryant soliton, which is (up
to homothety) the unique complete, rotationally symmetric steady gradient
soliton on $\R^{n+1}$. This behavior is unsurprising: as fixed points of
Ricci flow modulo diffeomorphism and scaling, solitons are expected to
provide natural models for its behavior near singularities.  Our
asymptotics confirm this expectation and give precise information on the
length and time scales on which the forward evolution is modeled by the
Bryant soliton.

It seems reasonable to expect that a solution will continue to exist if one
admits initial data that are small, possibly asymmetric, perturbations
of the data considered here. It also seems reasonable that a uniqueness
statement should hold, but proving that would likely require methods different
from those used in this paper.

\section{Recovering from a neckpinch singularity}
\label{sec:Setup}
\subsection{The initial metric and its regularizations}
\label{sec:initialData}

We will construct Ricci flow solutions starting from singular limit metrics
resulting from rotationally symmetric neckpinches.  For a review of such
metrics, see Appendix~\ref{sec:DifferentiateTheAsymptotics} and, in
particular, Lemma~\ref{cor:AK2-Differentiable}.

Let $\NP$ and $\SP$ be the north and south poles of the sphere $\cS^{n+1}$.
We identify the (doubly) punctured sphere $\cS^{n+1}\setminus\{\NP, \SP\}$
with $(-1, 1)\times \cS^n$.  On this punctured sphere, we then consider
initial metrics $g_0$ of the form
\begin{equation}
  \label{eq:GeneralInitialData}
  g_0 = \varphi_0(x)^2(dx)^2 + \psi_0(x)^2\gcan
\end{equation}
where $\varphi_0, \psi_0$ are smooth functions on $(-1, 1)$.
Table~\ref{tab:initialdata} lists the assumptions we make about the initial
metric. (These assumptions are satisfied by the singular limits of the
solutions studied in \cite{AK07}.) To make our assumptions geometric, we break
gauge invariance by choosing distance $s$ to the north pole as the
preferred coordinate.  In this coordinate, the initial metric then appears as
\begin{equation} \label{eq:InitialInitialData} g_0=(ds)^2+\psi_0(s)^2\gcan
\end{equation}
where $\ell$ is the distance between north and south poles in the metric
$g_0$, and where $\psi_0\in C^\infty(J)\cap C^1(\bar{J})$, with
$J=(0,\ell)$.

\begin{table}[bt] \vrule\vbox {\hrule
  \begin{align*}
    \mathrm{(M1)\;\;}&
    \psi_0(s) > 0 \text{ for all } s\in J \\
    \mathrm{(M2)\;\;}&
    \psi_0(0) = \psi_0(\ell) = 0 \\
    \mathrm{(M3)\;\;}&
    \psi_0'(\ell) = -1 \\
    \mathrm{(M4)\;\;}& \psi_0(s)^2 = \bigl(\tfrac{n-1}4 + o(1)\bigr)
    \frac{s^2}{-\log s}
    \quad (s\searrow 0)\\
    \mathrm{(M5)\;\;}& \psi_0(s)\psi_0'(s) = \bigl(\tfrac{n-1}4 +
    o(1)\bigr) \frac{s}{-\log s}
    \quad (s\searrow 0) \\
    \mathrm{(M6)\;\;}&
    |\psi_0'(s)| \le 1 \quad (0<s<\ell)\\
    \mathrm{(M7)\;\;}&
    \exists_{r_\#>0}\;\psi_0'(s) \ne 0 \text{ whenever }\psi_0(s)<2r_\#\\
    \mathrm{(M8)\;\;}& \exists_{\amax <\infty}\forall_{s\in J}\;
    |a_0(s)|\leq \amax \text{ (where $ a_0(s) = \psi_0\psi_0'' -
    {\psi_0'}^2 + 1$)}
  \end{align*}
  \hrule}\vrule
  \caption{Assumptions on the initial metric $g_0 = (ds)^2 +
  \psi_0(s)^2\gcan$.}
  \label{tab:initialdata}
\end{table}

\begin{figure}[b]
  \centering
  \includegraphics{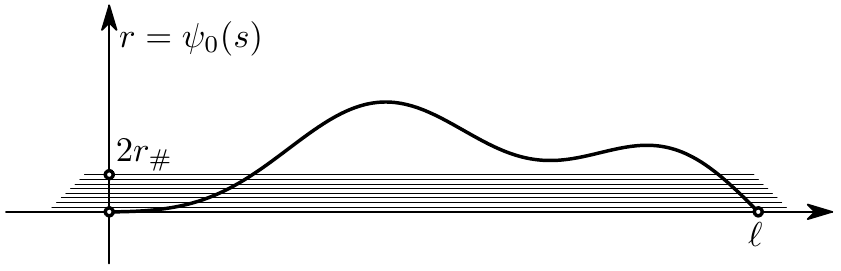}\caption{The initial data}
  \label{fig:initdata}
\end{figure}

Since the initial metric $g_0$ is singular at the north pole, the standard
short-time existence theory for Ricci flow does not provide a solution.
Because our initial metric is of the form \eqref{eq:InitialInitialData}
with $\psi_0(s) \sim s|\log s|^{-1/2}$ for $s\to 0$, the volume of a ball
with radius $s$ centered at the singular point is $(C+o(1)) s^{n+1}|\log
s|^{-n/2} = o(s^{n+1}) $.  So no continuous metric $\tg$ on $\cS^{n+1}$
exists for which $\exists c>0 : c \tg \le \gcan \le c^{-1}\tg$. Therefore,
neither the methods of Simon \cite{Simon02} nor of Koch--Lamm \cite{KL09}
can be used here to construct a solution.

Instead, we will construct a solution by regularizing the metric in a
small neighborhood of the north pole, yielding a smooth metric $g_{\omega}$
for each small $\omega>0$, with $g_\omega \to g_0$ as $\omega\searrow0$.
Each of these may be regarded as a ``surgically modified'' solution, obtained
by replacing the singularity with a smooth $(n+1)$-ball.
The short-time existence theorem then guarantees the existence of a
solution $g_{\omega}(t)$ on some time interval $0<t<T_{\omega}$, with
$g_{\omega}$ as initial data.  We will obtain a lower bound $T>0$ for
$T_{\omega}$ and show that a subsequence of the $g_{\omega}$ converges in
$C^{\infty}$ for all $0<t<T$.

The lower bound for $T_\omega$ follows easily from our previous
characterization of singular solutions in \cite{AK04}.  Obtaining
compactness of the solutions thus obtained turns out to be harder and will
consume most of our efforts in Section~\ref{sec:Compactness}.

In the following description of the regularized initial metric $g_\omega$,
we will have to refer to a number of constants and functions which are
properly defined in Sections \ref{sec:Formal} and \ref{sec:Informal}.  We
note that while we have phrased the description in terms of the
geometrically defined quantities $s, \psi, \psi_s$, it applies to metrics of
the form (\ref{eq:GeneralInitialData}).

For given small $\omega>0$, we split the manifold $\cS^{n+1}$ into two
disjoint parts, one of which is the small neighborhood $\cN_\omega$ of the
north pole in which $\psi_0(s) < \rho_*\sqrt \omega$.

On $\cS^{n+1}\setminus \cN_\omega$, we let our regularized metric $g_\omega$
coincide with the original metric $g_0$.  Within $\cN_\omega$, we let
$g_\omega$ be a metric of the form $g_\omega = (ds)^2 +
\psi_\omega(s)^2\gcan$, where $\psi_\omega$ is a monotone function, at
least when $\psi_\omega(s) \leq \rho_*\sqrt{\omega}$.  In $\cN_\omega$, we
can therefore choose $r=\psi_\omega(s)$ as coordinate. In this
coordinate, we require $g_\omega$ to be of the form
\begin{equation} \label{eq:regularized-metric} g_\omega =
  \frac{(dr)^2}{v_\omega(r)} + r^2\gcan,
\end{equation}
where $v_\omega(r)$ satisfies
\begin{equation}
  v^-(r, \omega) \leq v_\omega(r) \leq v^+(r, \omega).
  \label{eq:reg-metric-sandwich}
\end{equation}
Here $v^\pm(\cdot, \omega)$ are the sub- and super- solutions constructed
in Section~\ref{sec:Informal}, and evaluated at time $t=\omega$.

Recall that the scale-invariant difference of sectional curvatures,
\[
a\doteqdot(L-K)\psi^2\quad\left( =\psi\psi_{ss}-\psi_{s}^2+1\right)
\]
is uniformly bounded from above and below under Ricci flow
\cite[Corollary~3.2] {AK04}.  After possibly increasing the constant
$\amax$, we may assume that our regularized metrics $g_\omega$ satisfy
hypotheses (M1)--(M8) in Table~\ref{tab:initialdata}, with the exception
that near $s=0$ one has $\psi_{\omega}'(s) \to +1$ instead of (M4) and (M5)
\footnote{This actually follows from the fact that $\psi'(s) =
\sqrt{v(r,\omega)}$}.  Also, because we have modified the metric $g_0$
near the north pole, the distance $\ell_\omega$ between north and south
poles will not be exactly $\ell$, although we do have $\ell_\omega = \ell +
o(1)$ as $\omega\searrow0$.

Solving Ricci flow starting from $g_\omega$ produces a family of metrics
\[
g_{\omega}(t)=\varphi(x,t)^2(dx)^2+\psi(x,t)^2\gcan
\]
which evolve according to (\ref{eq:RF-psi}).  For these metrics, one will
then have both
\[
|\psi_{s}|\leq1
\]
and
\[
|a_{\omega}|=|\psi\psi_{ss}-\psi_{s}^2+1|\leq \amax.
\]
Therefore, again by (\ref{eq:RF-psi}),
\begin{equation}\label{eq:ddt-psi-squared-bound}
  |(\psi^2)_{t}|=|2a+2n(\psi_{s}^2-1)|\leq2(\amax+n)
\end{equation}
is bounded.  In particular, one has $(\psi^2)_{t}<2(\amax+n)$ at critical
points of $x\mapsto\psi(x,t)$.  Thus if one sets
\[
T_0=\frac{r_{\#}^2}{\amax+n},
\]
then condition~(M7) implies that for $0<t<T_0$, the function $x\mapsto
\psi(x,t)$ has no critical points with $\psi(x,t)\leq{r_{\#}}$.  In
particular, there can be no new neckpinch before $t=T_0$.  We have thus
proved
\begin{lemma}
  \label{lem:life-time-lower-bound}
  $\displaystyle T_\omega \geq T_0$ for all $\omega>0$.
\end{lemma}

\medskip

Since $s \mapsto \psi_\omega(s, t)$ is monotone when $\psi_\omega<r_\#$, a
function $r\mapsto u(r,t)$ is defined for $0<r<r_{\#}$ for which one has
\begin{equation}\label{eq:u-defined}
  \psi_{s}(x,t)=u(r,t)=u\bigl(\psi(x,t),t\bigr)
\end{equation}
for all $x$ on the northern cap of the sphere where $\psi\leq{r_{\#}}$
holds.  With respect to these local coordinates, one may write the
metric~(\ref{eq:InitialInitialData}) as
\[
g=\frac{(dr)^2}{u(r)^2}+r^2\gcan .
\]
The sectional curvatures $K$ and $L$ defined in (\ref{eq:Sectionals}) are
computed in these coordinates by $K=-uu_{r}/r$ and $L=(1-u^2)/r^2$.

\subsection{Evolution equations for $u(r,t)$ and $v(r,t)=u(r,t)^2$}

To derive a differential equation for $u$, we first recall that $\psi_s$
satisfies
\[
\partial_{t}\psi_{s}=\psi_{sss}+(n-2)\frac{\psi_{s}}{\psi}\psi_{ss}
+(n-1)\frac{\psi_{s}(1-\psi_{s}^2)}{\psi^2}.
\]
On the other hand, from the definition (\ref{eq:u-defined}) of $u$, we find
by the chain rule that
\[
\partial_{t}\psi_{s}=\psi_{t}u_{r}+u_{t}.
\]
Combining these two relations, we get $u_{t}=\partial_{t}\psi_{s}-\psi
_{r}u_{r}$ in terms of $\psi$ and its $s$ derivatives.  To rewrite this in
terms of $u$ and its derivatives, we use the fact that when $t$ is fixed,
it follows from $u=\psi_{s}$ that
\[
\frac{\partial}{\partial s}=u\frac{\partial}{\partial r},\quad\text{ and hence}
\quad\psi_{ss}=uu_{r},\quad\psi_{sss}=u\bigl(uu_{r}\bigr)_{r}.
\]
One then arrives at a parabolic \textsc{pde} for $u$, to wit,
\begin{equation}
  u_{t}=u\bigl(uu_{r}\bigr)_{r}-uu_{r}^2+\frac{n-1-u^2}{r}u_{r}
  +(n-1)\frac{u(1-u^2)}{r^2}.  \label{eq:u-pde}
\end{equation}

The quantity $v=u^2$ satisfies a similar but slightly simpler equation,
\begin{equation}
  v_{t}=\cF [v], \label{eq:v-t}
\end{equation}
where
\begin{equation}
  \cF [v] \doteqdot vv_{rr}-\tfrac12v_{r}^2+\frac{n-1-v}{r}v_{r}
  +\frac{2(n-1)}{r^2}v(1-v).  \label{eq:F-def}
\end{equation}
In terms of $v$, we have
\begin{equation}
  K = -\frac{v_r}{2r}, \qquad
  L = \frac{1-v}{r^2}.
  \label{eq:KandLfrom-v}
\end{equation}

\subsection{The shape of the singular initial data in the $r$ variables}

We now derive the asymptotic behavior of the unmodified function $v(r,t)$
as $r\searrow0$ at time $t=0$.  First we estimate $\uini(r)=u(r,0)$.  By
equation~(\ref{initial-psi}) of Lemma~\ref{cor:AK2-Differentiable} in
Appendix~\ref{sec:DifferentiateTheAsymptotics}, the quantity
$r(s)=\psi(s,0) $ satisfies
\begin{equation}
  r=\left[  \tfrac12\sqrt{n-1}+o(1)\right]  \frac{s}{\sqrt{-\log{s}}}
  \label{eq:r_asymptotics}
\end{equation}
as the arc-length variable $s\searrow0$.  If this were an exact equation,
one could solve it explicitly for $s$, yielding
\[
s=r\sqrt{\frac2{n-1}\;W\left( \frac{n-1}{2r^2}\right) },
\]
where $W$ denotes the Lambert-W (product-log) function, i.e.~the inverse of
$y\mapsto ye^{y}$ restricted to $(0,\infty)$.  For our purposes, it
suffices to notice that (\ref{eq:r_asymptotics}) implies that $\log
r=\left[ 1+o(1)\right] \log s$ as $s\searrow0$, hence that
\begin{equation}
  s=\left[  \frac2{\sqrt{n-1}}+o(1)\right]  r\sqrt{-\log r}
\end{equation}
as $r\searrow0$.  Again by Lemma~\ref{cor:AK2-Differentiable}, it is
permissible to differentiate (\ref{eq:r_asymptotics}), whereupon one finds
in equation~(\ref{initial-psi_s}) that the initial data $\uini(\cdot)=\psi
_{s}(\cdot,0)$ satisfies
\begin{equation}
  \uini(r)=\bigl[\frac{\sqrt{n-1}}2+o(1)\bigr]\frac1{\sqrt{-\log r}}
\end{equation}
as $r\searrow0$.  Squaring, we conclude that $v$ initially must satisfy
\begin{equation}\label{eq:vcusp}
  \vini(r)=\bigl(\frac{n-1}4+o(1)\bigr)\frac1{-\log r}\quad(r\searrow0).
\end{equation}

\section{Possible complete smooth solutions}
\label{Compact}

In this paper, we construct smooth solutions $g(t)$ of Ricci flow on the
compact sphere $\cS^{n+1}$. However, our initial data are only smooth on the
punctured sphere $\cS^{n+1}\setminus\{\NP\}$.  Inspired by
Topping's example in \cite{Topping09}, one could also look for solutions on
the punctured sphere which at all positive times $t$ represent complete
metrics.  Such solutions must necessarily be unbounded.  Below, we will show
that such solutions do not exist for our initial metric.  Consequently, all
possible smooth solutions starting from our singular initial metric extend
to the compact sphere.  Their sectional curvatures $K$ and $L$ must be
bounded; and in view of $L=(1-v)/r^2$, we find that the quantity $v(r,t)$
must satisfy
\begin{equation} \label{eq:qualitative-v-bound} |v(r, t) - 1| \leq \sup
  |L_{g(t)}| \; r^2.
\end{equation}
Thus
\begin{equation} \label{eq:v-boundary-condition} v(0, t) = 1
\end{equation}
is the only relevant boundary condition at $r=0$ for our problem.

Our existence proof of a solution starting from the singular initial metric
involves the construction of a family of sub- and super- solutions
$v^\pm_{\varepsilon,\delta}$, indexed by parameters $\varepsilon,
\delta>0$.  In this section, we show that the $v$ function corresponding to
any smooth solution must lie between the sub- and super- solutions
$v^{\pm}_{\varepsilon,\delta}$ to be constructed in Section~\ref{sec:Informal}.

\subsection{Lower barriers for $v$}
We first show that any such solution remains positive for small $r$ and
$t$, with an estimate that proves the following result.

\begin{theorem}
  \label{CompactResult}Any smooth complete rotationally symmetric
  forward Ricci flow evolution from a rotationally symmetric
  neckpinch singularity is compact with $v(0,t)=1$ for all $t>0$ that it exists.
\end{theorem}

Given $\lambda\in(0,1)$, define $\bv_{\lambda}(r)$ for $0\leq r<1$ by
$\bv_{\lambda}(0)=0$ and
\begin{equation}
  \bv_{\lambda}(r) \doteqdot\frac{\lambda(n-1)}4\frac1{-\log r}
  \qquad(0<r<1).
\end{equation}
For later use, we note that for $r\in(0,1)$, one has
\begin{gather}
  (\bv_{\lambda})_{r}=\frac{\lambda(n-1)}4\frac1{r(\log r)^2}, \\
  (\bv_{\lambda})_{rr}=\frac{\lambda(n-1)}4\left\{ \frac2{r^2(-\log
  r)^3}-\frac1{r^2(\log r)^2}\right\} .
\end{gather}

Theorem~\ref{CompactResult} is essentially a corollary of the following
observation, which shows that on a sufficiently short time interval,
$\bv_{\lambda}(r)-\varepsilon(1+t)$ is a subsolution of (\ref{eq:v-t}) for
all small $\varepsilon>0$.

\begin{lemma} \label{lem:Cristina} Let $v$ be any nonnegative smooth
  function satisfying
  \begin{equation} \label{eq:YetAnotherEquation} v_{t} \geq  \cF [v], \qquad
    v(r,0)\geq \vini(r),
  \end{equation}
  where $\cF $ is defined by (\ref{eq:F-def}), and $\vini$ satisfies
  (\ref{eq:vcusp}).

  Then $v\geq\bv_{\lambda}$ on $(0,r_0]\times[0,t_0]$ for some small
  positive $r_0$ and $t_0$, independent of the boundary condition
  $v(0,\cdot)$.
\end{lemma}
We note that the following argument also allows $v$ to be a piecewise
smooth supersolution whose graph only has concave corners, such as the
``glued'' supersolutions we construct in Section~\ref{sec:Informal}.  We will only
write the proof for the smooth case.
\begin{proof}
  Since $\lambda<1$, we may fix $0<r_0\ll1$ so small that $\bv
  _{\lambda}(r)<\min\left\{ \vini(r),\frac14\right\} $ and $(\log
  r)^2>(n-1)/4$ for $0<r<r_0$.  Then fix $t_0>0$ small enough so that
  $v(r,t) $ exists for $0<t\leq t_0$.  By making $t_0$ smaller if
  necessary, we may assume that $\bv_{\lambda}(r_0)<v(r_0,t)$ for $0\leq
  t\leq t_0$.

  Given $\varepsilon>0$, define
  \begin{equation}
    f_{\varepsilon}(r,t) \doteqdot v(r,t)-\bv_{\lambda}(r)+\varepsilon(1+t).
  \end{equation}
  We shall prove that $f_{\varepsilon}(r,t)>0$ for $0<r\leq r_0$.  The
  lemma will follow by letting $\varepsilon$ go to zero.

  To simplify the notation, we henceforth write $f\equiv f_{\varepsilon}$
  and $\bv \equiv\bv_{\lambda}$.  We also set $\mu\doteqdot\lambda(n-1)/4$.

  Observe that $f(r,0)\geq\varepsilon$ for $0<r\leq r_0$ and that both
  $f(0,t)\geq\varepsilon$ and $f(r_0,t)\geq\varepsilon$ hold for $0<t\leq
  t_0$.  To obtain a contradiction, suppose that there exists a first time
  $\bt \in(0,t_0]$ and a point $\br \in(0,r_0]$ such that $f(\br ,\bt )=0$.
  Then by (\ref{eq:YetAnotherEquation}), one has
  \begin{equation}
    0\geq f_{t}(\br ,\bt )=\cF \left[  v(\br ,\bt )\right]
    +\varepsilon.  \label{Contradiction}
  \end{equation}
  We claim that $\cF \left[ v(\br ,\bt )\right] \geq0$.  This contradicts
  (\ref{Contradiction}) and proves the lemma.

  To prove the claim, observe that
  \[
  0=f_{r}(\br ,\bt )=v_{r}(\br ,\bt )-\bv_{r}(\br
  )\qquad\text{and}\qquad0\leq f_{rr}(\br ,\bt )=v_{rr}(\br ,\bt
  )-\bv_{rr}(\br ,\bt ).
  \]
  Thus we obtain
  \begin{align*}
    \br^2\cF \left[ v(\br ,\bt )\right]
    &  =v\br^2v_{rr}-\frac12(\br v_{r})^2+(n-1-v)\br v_{r}+2(n-1)(1-v)v\\
    & \geq v\br^2\bv_{rr}-\frac12(\br \bv_{r})^2
    +(n-1-v)\br \bv_{r}+2(n-1)(1-v)v\\
    & =\frac{\mu}{(\log r)^2}\left\{ n-1-2v-\frac{\mu}{2(\log r)^2}\right\}\\
    &\hspace{20ex} +2v\left\{  (n-1)(1-v)+\frac{\mu}{(-\log r)^3}\right\} \\
    & \geq\frac{\mu}{2(\log r)^2}\left\{ 1-\frac{\mu}{(\log r)^2}\right\} \\
    & >0.
  \end{align*}
  Here we used the facts that $n\geq2$ and $0\leq v<\bv \leq1/4$ at
  $(\br,\bt)$.  This proves the claim and hence the lemma.
\end{proof}

Now we can prove the main result of this subsection.

\begin{proof}[Proof of Theorem~\ref{CompactResult}] Let $\cK_{r_0}$
  denote the complement in $\cS^{n+1}$ of the neighborhood of the north pole $\NP\in\cS^{n+1}$
  in which $\psi_\init<r_0$. For a smooth forward evolution, $\cK_{r_0}$ is
  precompact for small positive time. So the only way a noncompact solution
  could develop is if a neighborhood of the singularity immediately became
  infinitely long. But for all $r\in(0,r_0]$, Lemma~\ref{lem:Cristina} implies that
  \[
  \frac1{\sqrt{v(r,t)}}\leq\frac2{\sqrt{n-1}}\sqrt{-\log r}.
  \]
  It follows that
  \[
  d_{g(t)}(0,r_0)=\int_0^{r_0}\frac{\,dr\,}{\sqrt{v(r,t)}}<\infty,
  \]
  hence that the solution is compact.
  Because the sectional curvature of the metric $g$ on planes tangent to
  $\{r\}\times\cS^{n}$ is $L=(1-v)/r^2 $, the solution will be smooth only
  if $v(0,t)=1$.
\end{proof}

\subsection{Yet another maximum principle}

In this section, we prove that the $v$ function for any smooth, complete
forward evolution with $g_0$ as initial metric is trapped between the sub-
and super- solutions which we will construct in Section~\ref{sec:Informal}.
This implies the claims about the asymptotic profile of the solution in
Theorem~\ref{thm:main}.

We begin by establishing a suitable comparison principle.

\begin{lemma}
  \label{yamm}
  Let $v^-$ and $v^+$ be nonnegative sub- and super- solutions,
  respectively, of $v_t=\cF[v]$. Assume that either $v^-$ or $v^+$
  satisfies
  \begin{equation}\label{eq:yamm-v-conditions}
    v_{rr}\leq C,\quad
    K=-\frac{v_r}{2r}\leq C,\quad
    L=\frac{1-v}{r^2} \leq C
  \end{equation}
  for some constant $C<\infty$ on a compact space-time set $\Xi=[0,\br]
  \times [0,\bt]$.

  If $v^-(\br ,t)\leq v^+(\br,t)$ holds for $0\leq t\leq\bt $, and
  $v^-(r,0)\leq v^+(r,0)$ holds for $0\leq r\leq\br $, then $v^- \leq v^+$
  throughout $\Xi$.
\end{lemma}

In this lemma we assume that $v^\pm$ are smooth sub- and super- solutions.
However, the proof works without modifications in the case where $v^\pm$
are piecewise smooth, where the graph of $v^-$ only has convex corners, and
the graph of $v^+$ only has concave corners.  In our maximum principle
arguments, we shall only evaluate $v^\pm$ at ``points of first contact with
a given smooth solution'' which are necessarily smooth points of $v^\pm$.
Thus the hypothesis \eqref{eq:yamm-v-conditions} is to be satisfied at all
smooth points of $v^+$ or $v^-$; and, in particular, we do not intend to
interpret the second derivative $v_{rr}$ in \eqref{eq:yamm-v-conditions} in
the sense of distributions.
\begin{proof}
  We prove the Lemma assuming that $v^+$ satisfies
  \eqref{eq:yamm-v-conditions}.

  For $\lambda>0$ to be chosen later and any $\alpha>0$, define
  \begin{equation}
    z=e^{-\lambda t}(v^+-v^-)+\alpha.
  \end{equation}
  Then $z>0$ on the parabolic boundary of $\Xi$. We shall prove that
  $z>0$ in $\Xi$.  Because this implies that $v^+-v^->-\alpha
  e^{\lambda \bt }$ in $\Xi$, the lemma follows by letting $\alpha
  \searrow0$.

  Suppose there exists a first time $t\in(0,\bt ]$ and a point $r\in
  (0,\br)$ such that $z(r,t)=0$. Then $z_ t(r,t)\leq0$, and
  \begin{align*}
    v^+(r,t) =v^-(r,t)-\alpha e^{\lambda t},\quad v^+_ r(r,t) =v^-_
    r(r,t),\quad v^+_{rr}(r,t) \geq v^-_{rr}(r,t).
  \end{align*}
  Hence at $(r,t)$, one has $0\geq e^{\lambda t}z_ t$, where
  \begin{align*}
    e^{\lambda t}z_ t
    & =v^+_t - v^-_ t-\lambda(v^+-v^-)\\
    & =v^- [v^+_{rr}-v^-_{rr}] +(v^--v^+)\left\{
    \lambda-v^+_{rr}+\frac{v^+_ r} r-2(n-1)\frac{1-v^+-v^-}{r^2}
    \right\}\\
    & \geq (v^--v^+)\left\{ \lambda-v^+_{rr}+\frac{v^+_ r}
    r-2(n-1)\frac{1-v^+-v^-}{r^2} \right\}.
  \end{align*}
  Thus using the uniform bounds $v^+_{rr}\leq C$, $-v^+_r/r\leq C$, and
  $(1-v^+)/r^2\leq C$ on $\Xi$, together with $v^-\geq0$, one obtains
  \[
  0\geq e^{\lambda t}z_ t >\alpha e^{\lambda t}( \lambda-C) .
  \]
  This is a contradiction for any $\lambda>C$. The result follows.

  To prove the Lemma in the case that the subsolution $v^-$ satisfies
  \eqref{eq:yamm-v-conditions} one uses the fact that at a first zero $(r,
  t)$ of $z$ one has
  \[
  e^{\lambda t}z_ t = v^+ [ v^+_{rr} - v^-_{rr} ] + (v^--v^+)\left\{
  \lambda-v^-_{rr}+\frac{v^-_ r} r-2(n-1)\frac{1-v^+-v^-}{r^2} \right\} .
  \]
\end{proof}

Note that a different formulation of the lemma above is: if $v^+$ is a
supersolution of $v_t=\cF[v]$, then $v^+ + \alpha e^{\lambda t}$ is a
strict supersolution.

\begin{theorem} \label{pinch} Let $v$ denote any solution of the Cauchy
  problem
  \[
  v_ t =\cF [v],\qquad v(r,0) =\vini (r),
  \]
  that is smooth for a short time $0<t<t_1$. Let $v^{\pm}_{\varepsilon,\delta}$
  denote the sub- and super- solutions, depending on $\varepsilon, \delta>0$,
  that are constructed in Section~\ref{sec:Informal}.

  For all small enough $\varepsilon, \delta>0$, there exist
  $\br_{\varepsilon,\delta}, \bt_{\varepsilon,\delta}>0$ such that
  \begin{equation} \label{eq:any-v-between-vpm} v^-_{\varepsilon,\delta}(r,
    t) \leq v(r, t) \leq v^+_{\varepsilon,\delta}(r, t)
  \end{equation}
  for all $(r, t) \in \Xi = [0,\br] \times [0, \bt]$.

\end{theorem}

\begin{proof}
  Let $\varepsilon, \delta>0$ be given.

  Since $v$ assumes our initial data, we have $v(r, 0) =
  (1+o(1))v_0(r)$ ($r\searrow0$). Since the sub- and super- solutions
  are initially given by $v^\pm_{\varepsilon,\delta}(r, 0) = (1\pm\delta)
  v_0(r)$, we find that there is some $\br>0$ such that
  $v^-_{\varepsilon,\delta}(r, 0) < v(r, 0) <v^+_{\varepsilon,\delta}(r,
  0)$ for all $r\in (0, \br]$.

  The solution $v$ and the sub- and super- solutions
  $v^\pm_{\varepsilon,\delta}$ are smooth for $r>0$, so there is some
  $\bt>0$ for which $v^-_{\varepsilon,\delta}(\br, t) < v(\br, t) <
  v^+_{\varepsilon,\delta}(\br, t)$ holds for $0\le t \le \bt$.  After
  shrinking $\bt$ if needed, we may also assume that $(1-\delta)v_0(\br) <
  v(r, t) < (1+\delta)v_0(\br)$ holds for $0\le t\le \bt$.

  Lemma \ref{yamm} would now immediately provide us with the desired
  conclusion, but unfortunately neither $v$ nor
  $v^{\pm}_{\varepsilon,\delta}$ meet the requirements needed to apply that
  result.  We overcome this problem by comparing time translates of $v$ and
  $v^\pm_{\varepsilon,\delta}$.

  First we show that $v \geq v^-_{\varepsilon, \delta}$.  Translate the given
  solution in time by a small amount $\kappa>0$, i.e.~consider the
  function
  \[
  \tv(r, t) = v(r, t+\kappa)
  \]
  on the space-time domain $\Xi_\kappa = (0, \br) \times (0, \bt-\kappa)$.
  By Lemma~\ref{lem:Cristina}, we know that $v(r, t)\geq v^-_{\varepsilon,
  \delta}(\br, t) \geq (1-\delta)v_0(r)$ for $0<t<\bt$.  Since the
  time-translate $\tv$ is smooth, it does satisfy
  \eqref{eq:yamm-v-conditions}, and we may conclude that $\tv(r, t) \geq
  v^-(r, t)$ on $\Xi_\kappa$.  Letting $\kappa\searrow0$, we end up with
  $v(r, t) \geq v^-(r, t)$ on $\Xi$.

  Next, we argue that $v \leq v^+_{\varepsilon,\delta}$ on $\Xi$.  To this
  end we introduce
  \[
  v^*(r,t)=v^+_{\varepsilon,\delta}(r,\kappa+t).
  \]
  Since $v^+_{\varepsilon,\delta}$ is a supersolution that starts out with
  $v^+_{\varepsilon,\delta}(r,0) = (1+\delta)v_0(r)$,
  Lemma~\ref{lem:Cristina} tells us that $v^+_{\varepsilon,\delta}(r, t)
  \geq (1+\delta)v_0(r) > v(r, 0)$ for all $(r,t) \in \Xi$.  We already
  have $v^+_{\varepsilon,\delta}(\br, t) > (1+\delta)v_0(\br) > v(\br, t)$
  for $0<t\le \bt$. So, since $v^*$ is piecewise smooth and satisfies
  \eqref{eq:yamm-v-conditions}, we obtain $v(r, t) \leq v^*(r, t)$ on
  $\Xi_\kappa$. Letting $\kappa\searrow0$ then completes the proof by
  showing that $v\leq v^+_{\varepsilon,\delta}$ on $\Xi$.
\end{proof}

\section{Formal matched asymptotics}
\label{sec:Formal}

As a heuristic guide to what follows, we now construct an approximate
solution of (\ref{eq:v-t}) that emerges from initial data satisfying
$v(r,0) = \vini(r) =[1+o(1)]v_0(r)$ as $r\searrow0$, where
by~(\ref{eq:vcusp}), $v_0$ is given by (\ref{eq:define-v-tilde}).  By
Theorem~\ref{CompactResult}, we may restrict our attention to smooth
solutions satisfying the (incompatible) boundary condition $v(0,t)=1$.

To describe the asymptotic behavior of the formal solution for small $t>0$,
we split the $(r,t)$ plane into three regions, which we label
\textit{inner}, \textit{parabolic}, and \textit{outer}, as in the
introduction.

We shall describe the solution in these separate but overlapping regions,
working our way from the outer to the inner region.

\subsection{The outer region $(r\sim 1)$}\label{sec:formal-outer}

Away from $r=0$, we expect the solution to be smooth.  So a good
approximation of the solution at small $t>0$ should be
\[
v(r,t)\approx v_0(r)+tv_1(r),
\]
where $v_1 \doteqdot\cF [ v_0] $.  One computes that
\begin{align*}
  \cF[v_0] &=\frac{v_0}{r^2} \left\{ 2(n-1)(1-v_0) +\frac1{-\log r} \left(
  (n-1)-2v_0 +\frac{3 v_0}{-2\log r}\right)
  \right\} \\
  & =\frac{v_0}{r^2} \left\{ 2(n-1) + \cO \left( \frac 1{-\log r}\right)
  \right\}.
\end{align*}
So we set
\begin{equation}
  \label{eq:v_out-approximate}
  v_\OUT (r,t) \doteqdot
  {v}_0(r)\left[  1+2(n-1)\frac{t}{r^2} \right]
  =\frac{n-1}{-4\log r} \left[  1+2(n-1)\frac{t}{r^2}\right].
\end{equation}
This suggests new space and time variables
\begin{equation}
  \label{eq:Define-rho-tau}\rho=\frac{r}{\sqrt{t}},\qquad\tau= \log t.
\end{equation}
With respect to $\rho$, which is bounded away from zero in the outer
region, the outer approximation may be written in the form
\begin{equation}
  \label{eq:v_out-wrt-rho}
  v_\OUT (r, t)
  =\frac{n-1}{-2\tau} \left[  1+\frac{2(n-1)}{\rho^2}\right]
  +\cO \bigl(\tau^{-2}\bigr)
\end{equation}
as $t\searrow0$, or, equivalently, as $\tau\searrow-\infty$.

\subsection{The parabolic (intermediate) region $(r\sim \sqrt t)$}

With $\rho$ and $\tau$ given by (\ref{eq:Define-rho-tau}), we define $W$ by
\begin{equation}
  v(r,t)=\frac{W(\rho,-\log t)}{-\log t}=\frac{W(\rho,\tau)}{-\tau}.
  \label{eq:Define-W}
\end{equation}
Then it is straightforward to compute that $W$ satisfies
\begin{equation}
  W_{\tau}+\frac1{-\tau}\bigl\{W-\cQ_\PAR
  [W]\bigr\}=\cL_\PAR [W], \label{eq:W-PDE}
\end{equation}
where $\cL_\PAR $ is the first-order linear operator
\begin{equation}
  \cL_\PAR [W]\doteqdot\left(  \frac{n-1}{\rho}+\frac{\rho
  }2\right)  W_{\rho}+\frac{2(n-1)}{\rho^2}W, \label{eq:L_par}
\end{equation}
and $\cQ_\PAR $ is the quadratic form given by
\begin{equation}
  \cQ_\PAR [W]\doteqdot WW_{\rho\rho}-\frac12(W_{\rho
  })^2-\frac1{\rho}WW_{\rho}-\frac{2(n-1)}{\rho^2}W^2.  \label{eq:N_par}
\end{equation}
As $t\searrow0$ one has $\tau=\log t\to-\infty$.  So if the limit
$\lim_{t\searrow0}W(\rho,\tau)$ exists, equation (\ref{eq:W-PDE}) leads one
to expect it to be a function $W_0(\rho)$ which satisfies $\cL_\PAR
[W_0]=0$.  To get a better approximate solution, we can add correction
terms of the form $W_{i}(\rho )/(-\tau)^{i}$ and substitute in
\eqref{eq:W-PDE}.  In this way one finds a formal asymptotic expansion of
the form
\begin{equation}
  W (\rho,\tau)=W_0(\rho)+\frac{W_1(\rho)}{(-\tau
  )}+\frac{W_2(\rho)}{(-\tau)^2}+\cdots, \label{eq:W-expansion}
\end{equation}
in which the $W_{j}$ can be computed inductively from
\begin{equation}
  \cL_\PAR [W_0]=0,\quad\cL_{{\PAR
  }}[W_{j+1}]=(j+1)W_{j}-\sum_{i=0}^{j}\cQ
  _\PAR [W_{i},W_{j-i}].  \label{eq:W_1-ODE}
\end{equation}
Here $\cQ_\PAR [\phi,\psi]$ is the bilinear form corresponding to the
quadratic form $\cQ_\PAR $ defined above.

We will not use this expansion beyond the lowest order term, but it did
prompt us to look for the sub- and super- solutions which we find in
Section~\ref{sec:ParabolicRegion}.

Taking only the lowest order term, our approximate solution in the
parabolic region is
\begin{equation}
  v_\PAR (r,t)=\frac{W_0(\rho)}{-\tau},
\end{equation}
where $W_0$ is a solution of $\cL_\PAR [W_0]=0$.  The general solution of
$\cL_\PAR [W_0]=0$ is $W_0(\rho)=c_0[1+2(n-1)/\rho^2]$, which gives us
\[
v_\PAR (r, t)=\frac{c_0}{-\tau}\left[ 1+\frac{2(n-1)}{\rho^2}\right] .
\]
Matching with (\ref{eq:v_out-wrt-rho}) for fixed $\rho$ and $t\searrow0$
tells us that the constant $c_0$ should be $c_0=(n-1)/2$.  Thus we get
\begin{equation}
  W_0(\rho)=\frac{n-1}2\left[  1+\frac{2(n-1)}{\rho^2}\right]
  \label{eq:W_0}
\end{equation}
and
\[
v_\PAR (r, t)=\frac{n-1}{-2\tau}\left[ 1+\frac {2(n-1)}{\rho^2}\right] .
\]

For small $\rho$, more precisely for $\rho=\cO (1/\sqrt{-\tau})$, one
defines a new space variable $\sigma=\rho\sqrt{-\tau} $ in order to write
the approximation
\begin{equation} \label{eq:small-rho-expansion} v_\PAR (r, t)
  \approx\frac{(n-1)^2}{\sigma^2} + \frac{n-1}{-2\tau}.
\end{equation}

At first glance, it may be surprising that the approximate solution in the
intermediate region is found by solving the first-order equation
(\ref{eq:L_par}) rather than by finding a stationary solution of a
parabolic equation.  This is caused by the fact that the parabolic
\textsc{pde} $v_t = \cF[v]$ is degenerate when $v=0$.

If the solution $v$ were approximately self-similar with parabolic scaling,
then one would have $v(\sqrt{t}\rho,t)\approx U(\rho)$ for some
self-similarly-expanding solution $U $.  However, equation
(\ref{eq:v_out-wrt-rho}) shows that this is incompatible with the behavior
of $v$ near the ``outer boundary'' of the intermediate region.

Nonetheless, we remark that self-similarly-expanding Ricci solitons
$U=\sqrt{V}$ do exist.  These were first discovered by Bryant, in
unpublished work.  Each is a solution of the \textsc{ode}
\begin{equation}
  U^2U_{\rho\rho}+\left\{  \frac{n-1-U^2}{\rho}+\frac{\rho}2\right\}
  U_{\rho}+\frac{n-1}{\rho^2}\left(  1-U^2\right)  U=0
  \label{eq:expander-ode}
\end{equation}
derived from~(\ref{eq:u-pde}), but each emerges from initial data
corresponding to a singular conical metric,
\begin{equation}
  g=\frac{(dr)^2}{U_{\infty}^2}+r^2\gcan , \label{eq:cone}
\end{equation}
where $U_{\infty}>0$.  We prove this assertion in
Appendix~\ref{app:AntiParabolic}.

We shall see below that solutions to (\ref{eq:expander-ode}) with $U(0)=1$
and $U(\infty)=0$ do exist when the term $\frac12\rho U'$ in
(\ref{eq:expander-ode}) is absent.  These solutions correspond to the Bryant
steady soliton.

Existence of solutions emerging from (\ref{eq:cone}) also follows from
Simon's work \cite{Simon02}, at least when $U_\infty$ is close to 1; but
Simon does not require the hypothesis of $\textrm{SO}(n+1, \R)$ symmetry.

\subsection{The inner (slowly changing) region ($r\sim \sqrt{t/(-\log
t)}$) \label{sec:formal-inner}}

The formal solution in the parabolic scale $r=\cO (\sqrt{T-t})$ found above
becomes singular as $r\to0$; and, in particular, it does not satisfy the
boundary condition at $r=0$.  Thus we look for a ``boundary layer''\ at a
smaller scale which will reconcile the incompatible initial $v(r,0)=0$ and
boundary $v(0,t)=1$ $(t>0)$ conditions.  Our derivation of the formal
solution in the parabolic region suggests that the smaller length scale
should be $\sqrt{t/(-\log t)}$.  So we let
\begin{equation}
  \theta=\sqrt{\frac{t}{-\log t}},\qquad\sigma=\frac{r}{\theta}
  \label{eq:theta-sigma-defined}
\end{equation}
and define
\begin{equation}
  v(r,t)=V\bigl(\sigma,\theta\bigr).  \label{eq:V-inner-defined}
\end{equation}
Then $V$ satisfies the \textsc{pde}
\begin{equation}
  \theta\theta_{t}\bigl\{\theta V_{\theta}-\sigma V_{\sigma}\bigr\}=\cF
  _\IN[V], \label{eq:ExactEquationInnerRegion}
\end{equation}
where $\cF_{{\mathrm{in}}}$ is obtained by replacing $r$-derivatives in
$\cF $ with $\sigma$-derivatives, namely
\[
\cF_{{\mathrm{in}}}[V] \doteqdot VV_{\sigma\sigma}-\tfrac12(V_{\sigma}
)^2+\frac{n-1-V}{\sigma}V_{\sigma}+\frac{2(n-1)}{\sigma^2}(V-V^2).
\]
We will abuse notation and simply write $\cF$ for $\cF_\IN$.

We begin with the observation that $\theta\theta_{t}=\cO \bigl((-\log
t)^{-1}\bigr)=o(1)$ for small $t$, so that the crudest approximation of
(\ref{eq:ExactEquationInnerRegion}) is simply the equation $\cF [V ]=0$.
This \textsc{ode} admits a unique one-parameter family of complete
solutions satisfying $V (0)=0$ and $V (\infty)=0$.  These solutions are
given by
\[
V_0(\sigma)=\bry (k\sigma)\qquad(0<k<\infty),
\]
where $\bry $ is the \emph{Bryant steady soliton,} whose asymptotic
behavior is $\bry (\sigma)=\sigma^{-2}+o(\sigma^{-2})$ for $\sigma \to
\infty$.  These assertions are proved in Appendix~\ref{sec:Bryant}.

Equation (\ref{eq:ExactEquationInnerRegion}) suggests that this crude
approximation is off by a term of order $\theta\theta_{t}$, which prompts
us to look for approximate solutions of the form
\begin{equation}
  \label{eq:InnerApproximation}
  V (\sigma,t) = V_0(\sigma) + \theta\theta_{t}V_1(\sigma),
\end{equation}
Here $V_0(\sigma) = \bry (k\sigma)$ contains an unspecified constant $k$
whose value we will determine later by matching with the approximate
solution from the parabolic region.

To find an equation for $V_1$, we observe that the \textsc{lhs} and
\textsc{rhs} of (\ref{eq:ExactEquationInnerRegion}) applied to $V $ yield
\begin{equation}
  \theta\theta_{t}(\theta V_{\theta}-\sigma V_{\sigma}
  )=-\theta\theta_{t}\sigma V_0'(\sigma)+(\theta\theta_{t}
  )^2\left[  V_1(\sigma)-\sigma V_1'(\sigma
  )\right]  +\theta^3\theta_{tt}V_1(\sigma),
  \label{eq:inner-formal-lhs}
\end{equation}
and
\begin{equation}
  \cF [V_0+\theta\theta_{t}V_1]=\cF [V_0]+d\cF_{V_0}[\theta\theta_{t}V_1
  ]+o(\theta\theta_{t}V_1)=\theta\theta_{t}d\cF_{V
  _0}[V_1]+o(\theta\theta_{t}V_1),
  \label{eq:inner-formal-rhs}
\end{equation}
respectively.  Here $d\cF $ is the first variation of the nonlinear
operator $\cF $, defined by
\[
d\cF_{V}[W]=\left.  \frac{d\cF [V+\epsilon W]}{d\epsilon
}\right|_{\epsilon=0}.
\]
It is given by the ordinary differential operator
\begin{multline}
  d\cF_{V_0}= V_0(\sigma)\frac{d^2}{d\sigma^2} +\left[
  \frac{n-1-V_0(\sigma)}{\sigma}-V_0'(\sigma) \right]
  \frac{d}{d\sigma}\\
  +\left[ V_0''(\sigma) -\frac{V_0'(\sigma)}{\sigma}
  +\frac{2(n-1)}{\sigma^2} (1-2V_0(\sigma)) \right].
\end{multline}
We note that
\[
|\theta^3\theta_{tt}|+|\theta\theta_{t}|^2 =
o(\theta\theta_{t})\quad\text{as}\quad t\searrow0.
\]
So by keeping only the most significant terms in the \textsc{lhs} and
\textsc{rhs}, we find the following equation for $V_1$,
\begin{equation}
  (d\cF_{V_0})[V_1]=-\sigma V_0'(\sigma).  \label{eq:V1-def-equation}
\end{equation}
We will first solve this equation in the case $k=1$ (when $V _0=\bry $).
The general case then easily follows by rescaling.

\begin{lemma}
  \label{lem:bryant-next-term}The ordinary differential equation
  \[
  d\cF_{\bry }[\cry ]=-\sigma\bry '(\sigma)
  \]
  has a strictly positive solution $\cry :(0,\infty)\to \R_{+}$ that
  satisfies
  \begin{equation} \label{eq:cryant-asymptotics} \cry (\sigma)=\left\{
    \begin{array} [c]{cc}
      M\sigma^2+o(\sigma^2) & (\sigma\searrow0),\\
      n-1+o(1) & (\sigma\nearrow\infty),
    \end{array}
    \right.
  \end{equation}
  for some constant $M>0$.

  All other solutions of $d\cF_{\bry }[\cry ]=-\sigma \bry '(\sigma)$ that
  are bounded at $\sigma=0$ are given by ${\tilde{\cry}}(\sigma)=\cry
  (\sigma)-\lambda\sigma\cry '(\sigma)$ for an arbitrary $\lambda\in\R $.

  If $V_0(\sigma)=\bry (k\sigma)$ for any $k>0$, then
  \[
  V_1(\sigma)=\frac{\cry (k\sigma)}{k^2}
  \]
  is a solution of (\ref{eq:V1-def-equation}).
\end{lemma}

\begin{proof}
  We know that $\cF [\bry (k\sigma)]=0$ for every $k>0$; differentiating
  this equation with respect to $k$ and setting $k=1$, we find that the
  homogeneous equation $d\cF_{\bry }[\phi]=0$ has a solution
  \[
  \phi(\sigma)\doteqdot-\sigma\bry '(\sigma).
  \]
  Given that $\phi$ satisfies $d\cF_{\bry }[\phi]=0$, one can use the
  method of reduction of order to find a second (linearly independent)
  solution $\hat{\phi}$.  One finds the following asymptotic behavior of
  $\hat{\phi}$ for small and large $\sigma$:
  \[
  \hat{\phi}(\sigma)=
  \begin{cases}
    C\sigma^{-(n-1)}+o(\sigma^{-(n-1)}) & (\sigma\searrow0),\\
    \exp\left\{ -\sigma^2/2(n-1)+o(\sigma^2)\right\} & (\sigma\nearrow
    \infty).
  \end{cases}
  \]

  If $\cry_{p}$ is any particular solution of $d\cF_\bry[\cry_{p}] =
  -\sigma\bry '(\sigma)$, then the general solution to $d\cF_\bry[\cry] =
  -\sigma\bry'(\sigma)$ is
  \begin{equation}
    \cry_{g}(\sigma)=a\phi(\sigma)+b\hat{\phi}(\sigma)+\cry_{p}
    (\sigma).  \label{eq:V1-general-soln}
  \end{equation}
  To obtain a particular solution which is bounded at $\sigma=0$, we note
  that for small $\sigma$, the equation $d\cF_{\bry }[\cry ]=-\sigma\bry '$
  is to leading order
  \[
  \cry '' +\frac{n-2}{\sigma}\cry ' -\frac {2(n-1)}{\sigma^2}\cry \approx
  -\bry''(0)\sigma^2,
  \]
  where $\bry ''(0)<0$ by Lemma~\ref{lem:BabyBryant} in
  Appendix~\ref{sec:Bryant}.  From this, one finds that a solution $\cry
  _{p}$ exists for which
  \begin{equation}
    \cry_{p}(\sigma)=(K+o(1))\sigma^4\text{ as }\sigma\searrow0,\text{
    where }K=-\frac{\bry ''(0)}{2(n+3)}>0.
    \label{eq:C-hat-asymptotics-at-zero}
  \end{equation}

  Since $\hat{\phi}$ is not bounded at $\sigma=0$, the only solutions given
  in (\ref{eq:V1-general-soln}) which are bounded at $\sigma=0$ are those
  for which $b=0$.

  Near $\sigma=\infty$, the equation $d\cF_{\bry }[\cry ]=-\sigma\bry '$
  is, to leading order,
  \[
  \frac1{\sigma^2}\cry ''+\frac{n-1}{\sigma} \cry
  '+\frac{2(n-1)}{\sigma^2}\cry \approx\frac2 {\sigma^2}.
  \]
  One then finds that there also is a solution $\cry_{\infty}$, which
  satisfies
  \[
  \cry_{\infty}(\sigma)=n-1+o(1)\qquad(\sigma\to\infty).
  \]
  But $\cry_{\infty}(\sigma)-\cry_{p}(\sigma)$, being the difference of two
  particular solutions, is a linear combination of $\phi$ and $\hat{\phi
  }$.  Since $\phi(\sigma)=o(1)$ and $\hat{\phi}(\sigma)=o(1)$ as $\sigma
  \to\infty$, we conclude that one also has
  \begin{equation}
    \cry_{p}(\sigma)=n-1+o(1) \label{eq:C-hat-asymptotics-at-infty}
  \end{equation}
  as $\sigma\to\infty$.  So we see that the general solution of $d\cF_{\bry
  }[\cry ]=-\sigma\bry '(\sigma)$ which is bounded as $\sigma\to0$ is given
  by $\cry (\sigma)=\cry_{p}(\sigma)+a\phi(\sigma)$.  Setting $a=-\lambda$
  leads to the general solution described in the statement of the lemma.

  The particular solution $\cry_{p}(\sigma)$ is positive for small and for
  large $\sigma$.  The solution $-\phi(\sigma)=-\sigma\bry ' (\sigma)$ to
  the homogeneous equation is positive for all $\sigma>0$.  Hence if one
  chooses $\lambda$ sufficiently large the resulting solution $\cry
  =\cry_{p}(\sigma)-\lambda\sigma\bry '(\sigma)$ will be strictly positive
  for all $\sigma>0$.  This is the positive solution of $d\cF_{\bry }[\cry
  ]=-\sigma\bry '(\sigma)$ which was promised in the lemma.

  Finally, if $d\cF_{\bry }[\cry ]=-\sigma\bry'(\sigma)$ then one verifies
  by direct substitution that $V_0 (\sigma)=\bry (k\sigma)$ and
  $V_1(\sigma)=k^{-2}\cry (k\sigma)$ satisfy \eqref{eq:V1-def-equation}.
\end{proof}

We return our attention to the approximate solution $V (\sigma,t)$ in
\eqref{eq:InnerApproximation}.  Combining $V_0(\sigma) = (1+o(1))
(k\sigma)^{-2}$ and $V_1(\sigma)=n-1+o(1)$ for large $\sigma$ with the
observation that
\[
\theta\theta_{t}=\frac{1-\log t}{2(\log t)^2}=\frac{\frac12+o(1)}{-\log
t},\qquad(t\searrow0),
\]
one sees that for large $\sigma$ and small $t$ our approximate inner
solution satisfies
\[
V (\sigma,t)=V_0(\sigma)+\theta\theta_{t}V_1 (\sigma) =[1+o(1)] \left\{
\frac{1}{k^2\sigma^2}+\frac1{-2(n-1)k^2\log t} \right\} .
\]
If we try to match this with the ``small $\rho$ expansion''\
(\ref{eq:small-rho-expansion}) for our approximate solution in the
parabolic region, then we see that we should choose
\begin{equation}
  k=\frac1{n-1}.  \label{eq:the-k-that-matches}
\end{equation}

\section{Construction of the barriers}
\label{sec:Informal}

\subsection{Outline of the construction}
In this section, we will construct lower and upper barriers for the
parabolic \textsc{pde}
\[
v_{t}=\cF [v]=\frac1{r^2}\left\{ vr^2v_{rr}-\frac12
(rv_{r})^2+(n-1-v)rv_{r}+2(n-1)(1-v)v\right\} .
\]
These barriers will apply to initial data satisfying
\[
\vini(r)=[1+o(1)]v_0(r)\quad\text{as}\quad r\searrow0,
\]
where $v_0$ is the asymptotic approximation defined in
(\ref{eq:define-v-tilde}), namely
\[
v_0(r)=\frac{n-1}{-4\log r}.
\]
The barriers will be valid on a sufficiently small space-time region
\[
0<r<r_*,\qquad0<t<t_*.
\]
Note that $r_*$ will not exceed the quantity $r_{\#}$ from assumption~(M7)
concerning the initial metric.  (See Section~\ref{sec:initialData}.)

Because we will not be able to write down barriers that are defined on this
whole domain, our construction proceeds in two steps.
Theorems~\ref{thm:outerlayer-subsuper}, \ref{thm:paraboliclayer-subsuper},
and \ref{InnerResult} constitute the first step.  In this step, in
accordance with the matched asymptotic description of the solution in
Section~\ref{sec:Formal}, we will produce three sets of barriers, each in
its own domain. (See Table~\ref{tab:subandsupers}.)
\begin{table}[tb]
  \begin{tabular}
    [c]{lcc}
    \hline
    \bf Outer
    &\vrule depth 12pt height 18pt width 0pt
    $\displaystyle(1\pm\delta)v_0(r)+(1\pm\varepsilon)t\cF [(1\pm\delta)v_0(r)]$
    & $\displaystyle\rho_* \sqrt{t}\leq r\leq r_*$
    \\[1ex]
    \bf Parabolic
    & $\displaystyle(1\pm\gamma_\pm)
    \frac{W_0(\rho)}{-\tau} \pm\frac{B^2}{\tau^2\rho^4}$
    &\vrule depth 12pt height 18pt width 0pt
    $\displaystyle\frac{\sigma_*}{\sqrt{-\tau}}\leq\rho\leq3\rho_*$
    \\[1ex]
    \bf Inner
    &\vrule depth12pt height18pt width0pt
    $\displaystyle\bry (k_{\pm}\sigma)
    +(1\mp\varepsilon)\theta\theta_{t}k_{\pm}^2\cry (k_{\pm}\sigma)$
    & $\displaystyle0<\sigma \leq3\sigma_*$\\
    \hline
  \end{tabular}
  \medskip

  \caption{The sub- and super- solutions with their domains.  Here,
  $\delta$ and $\varepsilon$ are sufficiently small; $\rho_* =
  A\varepsilon^{-1/2}$, $\sigma_* = B\varepsilon^{-1/2}$ for certain
  constants $A,B$ depending only on $n$; and $\gamma_\pm$, $k_\pm$
  are given by \eqref{eq:gk-minus-choice} and
  \eqref{eq:gk-plus-choice}.}
  \label{tab:subandsupers}
\end{table}
Note that the domains overlap.  In all three cases, time is restricted to
$0<t<t_*$.  The parameters $r_*<r_{\#}$, $\rho_*$, $\sigma_*$, and $t_*$
will be defined during the construction.  Although the construction admits
free parameters $\gamma$, $\delta$, $\varepsilon$, and $k$, all but
$\delta$ and $\varepsilon$ will be fixed in the second (``gluing'') step.

After constructing separate barriers, we must ``glue'' them together in
order to make one pair of sub-/super- solutions.  For example, to glue the
subsolutions in the parabolic and outer regions, we define
\[
v^{-}(r,t)=
\begin{cases}
  v_\OUT^{-}(r,t) & 3\rho_*\sqrt{t}\leq r\leq r_*\\
  \mathstrut & \mathstrut\\
  v_\PAR^{-}(r,t) & \sigma_*\sqrt{-t/\log t}\leq r\leq \rho_*\sqrt{t}
\end{cases}
\]
and
\[
v^{-}(r,t)=\max\left\{ v_\OUT^{-}(r,t),v_\PAR
^{-}(r,t)\right\}
\]
in the overlap between the outer and parabolic regions, when
$\rho_{\ast }\sqrt{t}\leq r\leq3\rho_*\sqrt{t}$.  To be sure that this
construction yields a true subsolution, we will verify the following
``gluing condition'':
\begin{equation}
  v_\PAR (\rho_*\sqrt{t},t)>v_\OUT (\rho_*
  \sqrt{t},t)\text{ and }v_\PAR (3\rho_*\sqrt{t}
  ,t)<v_\OUT (3\rho_*\sqrt{t},t).  \label{sub-solution-by-gluing}
\end{equation}

\begin{figure}[h]
  \centering
  \includegraphics{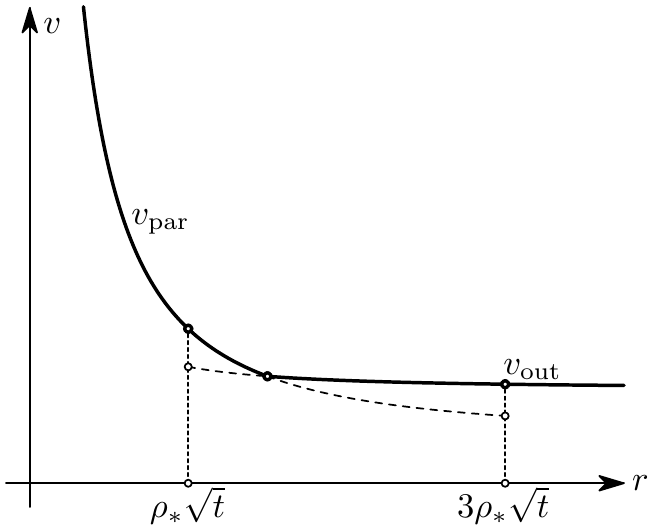}\caption{Gluing subsolutions}
\end{figure}

\noindent Lemmas~\ref{thm:can-glue} and \ref{thm:can-glue-too} constitute
the second step of the construction.  When this step is completed, we will
have chosen
\begin{equation}\label{eq:gk-minus-choice}
  \gamma_{-}=\delta+\frac{(1-\delta)\varepsilon^2}{\varepsilon+2A^2
  /(n-1)}\quad\text{and}\quad k_{-}=\frac1{n-1}\,\frac1{\sqrt{1-\gamma
  _{-}-\varepsilon/2(n-1)^2}}
\end{equation}
for subsolutions, and
\begin{equation}\label{eq:gk-plus-choice}
  \gamma_{+}=\delta+\frac{(1+\delta)\varepsilon^2}{\varepsilon+2A^2
  /(n-1)}\quad\text{and}\quad k_{+}=\frac1{n-1}\,\frac1{\sqrt{1+\gamma
  _{+}-\varepsilon/2(n-1)^2}}
\end{equation}
for supersolutions.  Note that for small $\delta,\varepsilon>0$, one has
$k_{-}>k_{+}$.  Because $\bry (\sigma)$ is decreasing, this implies in
particular that $V_\IN^{-}<V_\IN^{+}$ holds in the inner region.
\footnote{By taking $\varepsilon>0$ sufficiently small, depending on
$\delta>0$, one can make $k_{-}>(n-1)^{-1}>k_{+}$; but we will
not need this fact.}

\subsection{Barriers in the outer region}
\label{sec:OuterRegion}

In Section~\ref{sec:formal-outer}, we constructed an approximate solution of
the form $v(r,t) = v_0(r) + tv_1(r)$.  It turns out that a slight modification
of this approximate solution yields both sub- and super- solutions in the
outer region.

\begin{theorem}
  [Outer Region]\label{thm:outerlayer-subsuper} There exist positive
  constants $A=A(n)$, $r_*=r_*(n)$, and $t_*=t_*(n,\varepsilon)$ such that
  for all $|\delta|\leq\frac12$ and for all small $\varepsilon>0$, the
  functions
  \[
  v_\OUT^{\pm}(r,t) =(1\pm\delta)v_0(r) + (1\pm \varepsilon)t\cF
  [(1\pm\delta)v_0(r)]
  \]
  are sub- $(v_\OUT^{-})$ and super- $(v_\OUT^{+})$ solutions in the region
  \[
  \Omega_\OUT \doteqdot\left\{ (r,t):\rho_*\sqrt{t}\leq r\leq
  r_*,\quad0<t<t_* \right\},
  \]
  where $\rho_* = A/\sqrt{\varepsilon}$.
\end{theorem}

The constant $\delta$ need not be positive in this theorem, however, if one
wants a properly ordered pair of sub- and super- solutions, i.e.~if one
wants $v_\OUT^{-}<v_\OUT^{+}$, then one must choose $\delta>0$.

\begin{proof}
  We will show that $v_\OUT^{-}$ is a subsolution in a domain
  $\rho_*\sqrt{t}\leq r\leq r_*$, where $r_*\leq r_{\#}$ and $\rho_* \gg1$
  are to be chosen.  The proof that $v_\OUT^{+}$ is a supersolution is
  entirely analogous.

  To simplify notation, we define constants
  \[
  a \doteqdot(1-\delta)\tfrac{n-1}4,\qquad b \doteqdot1-\varepsilon,
  \]
  and functions
  \[
  \Lambda(r) \doteqdot\frac1{-\log r},\qquad\Gamma(r) \doteqdot\cF \left[
  a\Lambda(r)\right] .
  \]
  We henceforth write $v_\OUT^{-}(r,t)$ as
  \begin{equation}
    \label{eq:outer-sub-super-def}v_{a,b}(r,t) \doteqdot a\Lambda(r)+bt\Gamma
    (r).
  \end{equation}

  Observe that
  \[
  \Gamma(r)=\frac{a\Lambda}{r^2}\left\{ 2(n-1)+(n-1)(1-2a)\Lambda
  -2a\Lambda^2+\tfrac32a\Lambda^3\right\} .
  \]
  It follows that there exists $r_*>0$ such that
  \[
  \Gamma(r)>C\frac{\Lambda}{r^2}\quad\text{for}\quad0<r\leq r_*.
  \]

  Because $\partial_{t}(v_{a,b})=b\Gamma$, we have
  \[
  \partial_{t}(v_{a,b})-\cF [v_{a,b}]=(b-1)\Gamma-\left\{ \cF
  [a\Lambda+bt\Gamma]-\cF [a\Lambda]\right\} .
  \]
  Here the first term $(b-1)\Gamma$ has a sign for $0<r<r_*$.  The other
  term vanishes for $t=0$.  So for small $t$, continuity of $\partial_{t}
  (v_{a,b})-\cF [v_{a,b}]$ implies that $\partial_{t}(v_{a,b})- \cF
  [v_{a,b}]<0$ if $b<1$.  However, the size of the time interval on which
  this holds will depend on $r$ and, in particular, will shrink to zero as
  $r\to0$.  We now make this argument quantitative.

  Using the splitting of $\cF $ into linear and quadratic parts, one finds
  that
  \begin{align*}
    \cF [a\Lambda+bt\Gamma]-\cF [a\Lambda] & =bt\cL
    [\Gamma]+\cQ  (a\Lambda+bt\Gamma)-\cQ  (a\Lambda)\\
    & =bt\left\{ \cL [\Gamma]+\cQ (2a\Lambda+bt\Gamma ,\ \Gamma)\right\}
  \end{align*}
  where the bilinear form $\cQ (\cdot,\cdot)$ is defined by polarization
  from the quadratic form $\cQ (\cdot)$, i.e.~via
  \[
  \cQ (f,g) \doteqdot\tfrac14\left\{ Q(f+g)-Q(f-g)\right\} .
  \]
  Both $\cL [f]$ and $\cQ (f,g)$ are conveniently estimated in terms of the
  pointwise semi-norm
  \[
  [ f]_2(r) \doteqdot\left| f(r)\right| +r\left| f_{r}(r)\right| +r^2\left|
  f_{rr}(r)\right| .
  \]
  Indeed, for $0<r<r_*$,
  \[
  \left| \cL f\right| \leq\frac{C}{r^2}[f]_2\qquad\text{and} \qquad\left|
  \cQ (f,g)\right| \leq\frac{C}{r^2 }[f]_2[g]_2.
  \]
  So for all $r<r_*$, one has
  \[
  [ a\Lambda]_2\leq C\Lambda\qquad\text{and}\qquad[\Gamma]_2\leq
  C\frac{\Lambda}{r^2}.
  \]
  Therefore on the region $\rho_*\sqrt{t}\leq r\leq r_*$, one has
  \[
  t[\Gamma]_2\leq\frac{C}{\rho_*^2}\Lambda.
  \]
  For $\rho_*>1$, this implies that
  \begin{align*}
    \left| \cL [\Gamma]+\cQ (2a\Lambda+bt\Gamma ,\ \Gamma)\right| & \leq
    C\Lambda r^{-4}+C\Lambda^2r^{-4} +C\Lambda
    ^2r^{-4}\rho_*^{-2}\\
    & \leq C\Lambda r^{-4}.
  \end{align*}
  Observing that $\Gamma\geq C\Lambda r^{-2}$ for $0<r<r_*$, we thus
  estimate
  \[
  t\left| \cL [\Gamma]+\cQ (2a\Lambda+bt\Gamma ,\ \Gamma)\right|
  \leq\frac{C}{\rho_*^2}\frac{\Lambda}{r^2} \leq \frac{C}{\rho_*^2}\Gamma.
  \]

  Returning to our estimate for $\partial_{t}(v_{a,b})-\cF [v_{a,b}]$, we
  now have
  \[
  \partial_{t}(v_{a,b})-\cF [v_{a,b}]\leq(b-1)\Gamma+\frac{C}
  {\rho_*^2}\Gamma.
  \]
  Because $\Gamma>0$, we may conclude that $v_{a,b}$ is a subsolution in
  the outer region $\rho_*\sqrt{t}\leq r\leq r_*$ provided that
  \[
  \rho_*\geq C(1-b)^{-1/2}.
  \]
  Because $1-b=\varepsilon>0$, the theorem follows by taking $A=C$.
\end{proof}

\subsection{Barriers in the parabolic region}
\label{sec:ParabolicRegion}

In the parabolic region, we use the similarity variables $\rho$, $\tau$, and
$W$ defined in (\ref{eq:Define-rho-tau}) and (\ref{eq:Define-W}).
According to (\ref{eq:W-PDE}) the function $W$ satisfies $\cD_\PAR [W]=0$,
where
\[
\cD_\PAR [W] \doteqdot W_{\tau}-\cL_\PAR [W]+\frac1{-\tau} \{W-\cQ_\PAR
[W]\}.
\]
Here $\cL_\PAR $ is as in (\ref{eq:L_par}) while $\cQ_\PAR [W]$ is the same
quadratic differential polynomial as $\cQ [v]$, but with all
$r$-derivatives replaced by $\rho $-derivatives, namely
\[
\cQ_\PAR [W] \doteqdot \frac1{\rho^2} \bigl\{ W\rho ^2W_{\rho\rho} -
\tfrac12(\rho W_\rho)^2 - \rho W_\rho W-2(n-1)W^2 \bigr\}.
\]

\begin{theorem}[Parabolic region]\label{thm:paraboliclayer-subsuper}

  Let $\varepsilon>0$, $\rho_* = A/\sqrt\varepsilon_*$ be as in
  Theorem~\ref{thm:outerlayer-subsuper}.

  There exist $B=B(n,\varepsilon)>1$, and $t_*=t_*(n,\varepsilon)$ such
  that for any $\gamma$ with $0\leq\gamma\leq\frac12$, the functions
  \[
  v_\PAR^{\pm}(r,t)\doteqdot\frac{W_\PAR^{\pm}(\rho,\tau )}{-\tau}
  \]
  are sub- $(v_\PAR^{-})$ and super- $(v_\PAR^{+})$ solutions in the region
  \[
  \Omega_\PAR \doteqdot\left\{ (r,t):\frac{B}{\sqrt{\varepsilon}}
  \sqrt{\frac{t}{-\log t}}\leq r\leq3\rho_*\sqrt{t},\quad0<t<t_{\ast
  }\right\} .
  \]
  Here
  \[
  W_\PAR^{\pm}(\rho,\tau)\doteqdot(1\pm\gamma)W_0(\rho
  )\pm\frac{B^2}{-\tau\rho^4},
  \]
  where $W_0(\rho)$ is as in (\ref{eq:W_0}).
\end{theorem}

\begin{proof}
  We consider the case of a subsolution.  Recall that $cW_0$ is the general
  solution of the first-order \textsc{ode} $\cL_\PAR [W]=0$.  Moreover,
  \begin{equation}
    \label{eq:lpar-of-rhominus4}\cL_\PAR [\rho^{-4}] =
    -2\rho^{-6}(n-1+\rho^2) \leq-2 (n-1)\rho^{-6}.
  \end{equation}
  To show that $W_\PAR^{-}$ is a subsolution, we will verify that
  $\cD_\PAR[W_\PAR^{-}]<0$ on $\Omega_\PAR$.  To simplify notation, we
  write $W\equiv W_\PAR^{-}$ for the remainder of the proof.

  Observe that
  \[
  [W]_2 \doteqdot|W| +\rho| W_{\rho}| +\rho^2|W_{\rho\rho}| \leq C\left(
  \rho^{-2} + B^2|\tau|^{-1}\rho^{-4} \right)
  \]
  if $\rho\leq3\rho_*=3A/\sqrt{\varepsilon}$, where $C=C(n,A,\varepsilon)$.
  In the parabolic region $\Omega_\PAR $ defined above, one has
  \begin{equation}
    \label{eq:in-omega-par}0< \frac{B}{(-\tau)\rho^2} \leq\varepsilon,
  \end{equation}
  whence we get
  \[
  [W]_2 \leq C\rho^{-2}
  \]
  and also
  \[
  |\cQ_\PAR (W)| \leq\frac{C}{\rho^2} [W]_2^2 \leq \frac{C'}{\rho^6}.
  \]
  Now we compute that
  \[
  (-\tau)\cD_\PAR [W] = - \frac{2B^2} {(-\tau)\rho^4} + B^2\cL_\PAR
  [\rho^{-4}] + (1-\gamma)W_0 - \cQ_\PAR [W].
  \]
  Assuming that $\tau\leq-1$, we find, using (\ref{eq:lpar-of-rhominus4}),
  (\ref{eq:in-omega-par}), and also $\rho\leq3\rho_*=3A/\sqrt{\varepsilon}$,
  that
  \[
  (-\tau)\cD_\PAR [W] \leq\frac{C-2 (n-1) B^2}{\rho^6}
  \]
  in the parabolic region.  Hence we conclude that the function $W\equiv
  W_\PAR^{-}$ will indeed be a subsolution provided that $B^2 \geq\max\{ 1,
  C/2 (n-1) \} $.

  The left- and right- end points of the parabolic region at any time
  $\tau$ are given by $\rho=B/\sqrt{\varepsilon(-\tau)}$ and $\rho=3\rho_*
  =3A/\sqrt {\varepsilon}$, respectively.  So this region will be nonempty
  if $-\infty <\tau<-B^2/(3A)^2$.  Thus we choose $\tau_*=-\max\left\{
  1,B^2 /(3A)^2\right\} $.

  Construction of supersolutions $W_\PAR^{+}$ is similar.
\end{proof}

\subsection{Barriers in the inner region}
\label{sec:InnerRegion}

In the inner region, we work with the space and time variables $\sigma$,
$\theta$ defined in \eqref{eq:theta-sigma-defined}.  We consider $V(\sigma
,\theta)=v(r,t)$, as in (\ref{eq:V-inner-defined}).  Then, according to
(\ref{eq:ExactEquationInnerRegion}), Ricci flow is equivalent to $\cD_\IN
[V] =0$, where
\begin{align*}
  \cD_\IN [V] & \doteqdot\theta\theta_{t}\left( \theta
  V_{\theta}-\sigma V_{\sigma}\right)  -\cF [V]\\
  & = \theta^2 V_{t} - \theta\theta_{t}\sigma V_{\sigma}- \cF [V].
\end{align*}
The formal solution we found in Section~\ref{sec:formal-inner} is of the
form
\[
V (\sigma, t) = V_0(\sigma) + \lambda\theta\theta_{t} V_1(\sigma),
\]
where
\[
V_0(\sigma) = \bry (k\sigma),\text{ and } V_1(\sigma) = k^{-2}\cry
(k\sigma).
\]
In Section~\ref{sec:formal-inner}, we chose $\lambda=1$ and $k= 1/(n-1)$ in
order to match this solution with the formal solution in the parabolic
region.  Here we will show that small variations in $k$ and $\lambda$ lead
to sub- and super- solutions.

\begin{theorem}[Inner Region]\label{InnerResult}
  Let $\varepsilon$ and $B$ be as before.  There exists $t_* = t_*(n,
  \varepsilon, B)$ such that for any $k\in[\frac 1{2(n-1)}, \frac2{n-1}]$,
  the functions
  \[
  v_\IN^\pm(r, t) = V_\IN^{\pm}(\sigma,\theta) \doteqdot \bry (k\sigma) +
  (1\mp \varepsilon) \theta\theta_{t} k^{-2}\cry (k\sigma)
  \]
  are sub- $(v_\IN^{-})$ and super- $(v_\IN^{+})$ solutions in the region
  \[
  \Omega_\IN \doteqdot \left\{ (r,t):0<r\leq3\frac{B}{\sqrt {\varepsilon}}
  \sqrt{\frac{t}{-\log t}},\quad0<t<t_*\right\} .
  \]
\end{theorem}

We draw the reader's attention to the fact that $\theta\theta_{t} \cry
(k\sigma)>0$.  So for fixed $k$, the subsolution $V_{\IN }^{-}$ is larger
than the formal solution (which has $\varepsilon=0$), while the
supersolution $V_\IN^{+}$ is smaller.  To get a properly ordered pair of
sub- and super- solutions, we must (and can) choose $V_\IN^{+}$ and
$V_\IN^{-}$ with different values of $k$.

\begin{proof}
  We will prove that
  \[
  V_\IN^{-}=V_0(\sigma)+(1+\varepsilon)\theta\theta _{t}V_1(\sigma), \text{
  with }V_0(\sigma)=\bry (k\sigma)\text{ and }V_1(\sigma)=k^2\cry
  (k\sigma),
  \]
  is a subsolution in the region $\Omega_\IN $, i.e.~for
  \[
  0<\sigma<3\sigma_*=3B/\sqrt{\varepsilon},\qquad0<t<t_*,
  \]
  where $t_*$ is suitably chosen.  The proof that $V_\IN^{+}$ is a
  supersolution is similar.

  Upon substitution, we find that
  \begin{equation}
    \cD_\IN [V_\IN^{-}]=\theta^2(1+\varepsilon
    )(\theta\theta_{t})_{t}V_1-\theta\theta_{t}\sigma V
    _0'-(1+\varepsilon)(\theta\theta_{t})^2\sigma V_1
    '-\cF \bigl[V_0+(1+\varepsilon)\theta\theta_{t}
    V_1\bigr]. \label{eq:Vin-substituted}
  \end{equation}
  We can expand the last term, keeping in mind that $\cF [V _0]=0$, and
  that $\cF [V]$ is a quadratic polynomial in $V$ and its derivatives.  We
  get
  \[
  \cF \bigl[V_0+(1+\varepsilon)\theta\theta_{t}V
  _1\bigr]=(1+\varepsilon)\theta\theta_{t}d\cF_{V_0}
  [V_1]+(1+\varepsilon)^2(\theta\theta_{t})^2\cQ [V_1],
  \]
  where
  \[
  \cQ [V]=VV_{\sigma\sigma}-\tfrac12(V_{\sigma})^2-\tfrac1
  {\sigma}VV_{\sigma}-\tfrac{2(n-1)}{\sigma^2}V^2
  \]
  is the quadratic part of $\cF [V]$.  Applying this expansion to
  (\ref{eq:Vin-substituted}), we find that
  \begin{equation}
    \cD_\IN [V_\IN^{-}]=
    \varepsilon\theta\theta_{t}\sigma V_0'
    + (1+\varepsilon) \theta^2(\theta\theta_t)_t V_1
    - (1+\varepsilon)(\theta\theta_t)^2\sigma V_1'
    - (1+\varepsilon)^2(\theta\theta_t)^2\cQ [V_1]
    \label{eq:Vin-substituted2}
  \end{equation}
  The key to our argument is that the first term dominates the others on
  the interval $0<\sigma<\sigma_*=3B/\sqrt{\varepsilon}$.

  The facts from Lemma~\ref{lem:BabyBryant} that the Bryant soliton $\bry
  (\cdot)$ is strictly decreasing, and that it is given by $\bry
  (\sigma)=1+b_2\sigma^2+\cdots$ for small $\sigma$, with $b_2<0$, tell us
  that there is a constant $\eta>0$ such that
  \begin{equation}
    -\sigma V_0'(\sigma) \ge \eta\sigma^2
    \text{ for }\sigma\in(0,3\sigma_*).
    \label{eq:sig-Vsig-lowerbound}
  \end{equation}
  The asymptotics of $\cry (\sigma)$ both at $\sigma=0$ and $\sigma=\infty$
  from Lemma~\ref{lem:bryant-next-term} tell us that for some
  $C=C(n,\cry)<\infty$, one has
  \begin{equation}
    |V_1|+|\sigma V_1'|+|\cQ [V
    _1]|\leq C\sigma^2\text{ for all }\sigma\in(0,3\sigma_*),
    \label{eq:inner-other-terms-upper-bound}
  \end{equation}
  provided that $\sigma_*>1$.  Finally, by direct computation, one finds
  that
  \[
  \theta\theta_{t}=\tfrac12\bigl(\frac1{-\log t}+\frac1{(-\log t)^2
  }\bigr),\qquad\theta^2(\theta\theta_{t})_{t}=\frac{1+\frac2{-\log t}
  }{2(-\log t)^3},
  \]
  so that $|\theta^2(\theta\theta_{t})_{t}|\leq C(\theta\theta_{t})^3$ for
  some $C<\infty$, and for small $t$.

  Together with (\ref{eq:sig-Vsig-lowerbound}) and
  (\ref{eq:inner-other-terms-upper-bound}), we find that
  \begin{equation}
    \cD [V_\IN ]\leq\left\{  -\eta\varepsilon+C\theta\theta
    _{t}\right\}  \theta\theta_{t}\sigma^2 \label{eq:Vin-substituted3}
  \end{equation}
  for all $\sigma\in(0,3\sigma_*)$.  Since $\theta\theta_{t}=o(1)$ as
  $t\searrow0$, we find that $V_\IN^{-}$ is indeed a supersolution for
  small enough $t$.
\end{proof}

\subsection{Gluing the outer and parabolic barriers}
\label{sec:GlueOuterToParabolic}

The barriers $W_\PAR^{\pm}$ constructed in
Section~\ref{sec:ParabolicRegion} generate sub- and super- solutions
\[
v_\PAR^{\pm}=\frac{W_\PAR^{\pm}}{-\log t}
\]
for the original equation $v_{t}=\cF [v]$ in the parabolic region.

\begin{lemma}
  \label{thm:can-glue}
  Let $A, \epsilon$ and $\delta$ be as before, and set
  \begin{equation} \label{eq:gamma-specified} \gamma=
    \gamma_-(\delta,\varepsilon) = \delta+ \frac{(1-\delta)\varepsilon^2}
    {\varepsilon+ 2A^2/(n-1)}.
  \end{equation}
  If $-\tau_*$ is sufficiently large, then the gluing condition
  (\ref{sub-solution-by-gluing}) is satisfied for all $\tau<\tau_*$.
\end{lemma}

\begin{proof}
  We will verify the relations
  \[
  v_\PAR (\rho_*\sqrt{t},t)>v_\OUT (\rho_* \sqrt{t},t)\text{ and }v_\PAR
  (3\rho_*\sqrt{t} ,t)<v_\OUT (3\rho_*\sqrt{t},t)
  \]
  in the $W(\rho,\tau)$ rather than the $v(r,t)$ notation.

  When written in the $(\rho,\tau)$ variables, the subsolutions from the
  outer region take the form
  \[
  W_\OUT^{-}(\rho,\tau)\ =(-\tau)a{\Lambda}\left\{ 1+(1-\varepsilon
  )\frac{2(n-1)+\cO (\Lambda)}{\rho^2}\right\} ,
  \]
  where
  \[
  a \doteqdot(1-\delta)\frac{n-1}4.
  \]
  If $\rho_*\leq\rho\leq3\rho_*$, then
  \[
  \Lambda=\frac1{-\log r}=\frac1{-\log\rho-\frac12\log t}=-2\tau ^{-1}+\cO
  (\tau^{-2}).
  \]
  Hence
  \[
  W_\OUT^{-}(\rho,\tau)= (1-\delta)\frac{n-1}2 \left\{
  1+\frac{2(n-1)(1-\varepsilon)}{\rho^2} \right\} +\cO (\tau
  ^{-1}),\quad(\tau\to-\infty),
  \]
  uniformly in $\rho_*\leq\rho\leq3\rho_*$.

  The subsolutions from the parabolic region satisfy
  \[
  W_\PAR^{-}(\rho,\tau)= (1-\gamma)\frac{n-1}2 \left\{
  1+\frac{2(n-1)}{\rho^2}\right\} +\cO (\tau^{-1}),\quad
  (\tau\to-\infty).
  \]
  The outer and parabolic subsolutions have limits as $\tau\to-\infty$,
  namely
  \begin{align*}
    W_\OUT^{-\infty}(\rho) & \doteqdot(1-\delta)\frac{n-1}
    2\left\{  1+(1-\varepsilon)\frac{2(n-1)}{\rho^2}\right\}, \\
    W_\PAR^{-\infty}(\rho) & \doteqdot(1-\gamma)\frac{n-1} 2\left\{
    1+\frac{2(n-1)}{\rho^2}\right\} ,
  \end{align*}
  respectively.  At $\rho=0,\infty$, one finds that
  \begin{equation}
    \frac{W_\OUT^{-\infty}(\rho)}{W_\PAR^{-\infty}(\rho
    )}\longrightarrow
    \begin{cases}
      \dfrac{1-\delta}{1-\gamma} & \text{ as $\rho\to\infty$,}\\
      (1-\varepsilon)\dfrac{1-\delta}{1-\gamma} & \text{ as $\rho\to0$.}
    \end{cases}
    \label{eq:out-par-ratio}
  \end{equation}
  Requiring $W_\OUT^{-\infty}(2\rho_*)=W_{\PAR
  }^{-\infty}(2\rho_*)$ leads to \eqref{eq:gamma-specified}.  (Use
  $\rho_*=A/\sqrt{\varepsilon}$.) If \eqref{eq:gamma-specified} holds, then
  $\rho=2\rho_*$ is the only solution of $W_\OUT^{-\infty}
  (\rho)=W_\PAR^{-\infty}(\rho)$, and it follows from
  (\ref{eq:out-par-ratio}) that
  \[
  W_\OUT^{-\infty}(\rho_*)<W_\PAR^{-\infty} (\rho_*)\text{ and
  }W_\OUT^{-\infty}(3\rho_*)>W_\PAR^{-\infty}(3\rho_*).
  \]
  In particular, the gluing condition~(\ref{sub-solution-by-gluing}) is
  met.

  Because $W_\OUT^{-}$ and $W_\PAR^{-}$ are small ($\cO (\tau^{-1})$)
  perturbations of $W_\OUT^{-\infty}\ $ and $W_\PAR^{-\infty}$,
  respectively, these inequalities will continue to hold for all
  sufficiently large $-\tau$.
\end{proof}

A similar statement holds true for supersolutions.

\subsection{Gluing the inner and parabolic barriers}
\label{sec:GlueInnerToParabolic}

Recall that the inner and parabolic regions are
\begin{align*}
  \Omega_\IN & =\left\{ 0<\sigma\leq3\sigma_*,\quad
  0<t<t_*\right\}  ,\\
  \Omega_\PAR & =\left\{ \sigma_* \leq \sigma \leq 3A\sqrt{-\tau
  }/\sqrt{\varepsilon}, \quad0<t<t_* \right\} ,
\end{align*}
respectively, where
\[
\sigma_*=B/\sqrt{\varepsilon}.
\]
We now verify the ``gluing condition'' between the inner and parabolic
regions.

\begin{lemma} \label{thm:can-glue-too} If
  \begin{equation}
    k = k_-(\delta, \varepsilon)
    =\frac1{n-1}\,\frac1{\sqrt{1-\gamma-\varepsilon/2(n-1)^2}}
    \label{eq:k-inner-choice}
  \end{equation}
  and if $B$ is sufficiently large (depending only on $n$), then the gluing
  conditions
  \begin{equation} \label{eq:inner-parabolic-glue} V_\IN^-(\sigma_*,\tau) >
    V_\PAR^-(\sigma_* ,\tau) \text{ and }
    V_\IN^{-}(3\sigma_*,\tau)<V_\PAR^-(3\sigma_*,\tau)
  \end{equation}
  are satisfied for all $\tau<\tau_*$, provided that $-\tau_*$ is
  sufficiently large.
\end{lemma}

\begin{proof}
  By Theorem~\ref{InnerResult}, the subsolutions in the inner region
  $\Omega_\IN $ have a limit as $t\searrow0$, equivalently, as
  $\tau\searrow-\infty$.  The limit is
  \begin{equation} \label{eq:SmallTime1} V_\IN^{-\infty}(\sigma) =
    \lim_{t\searrow0} V_\IN^{-}(\sigma,t) = \bry(k\sigma).
  \end{equation}
  Our asymptotic expansion of the Bryant soliton (Lemma
  \ref{lem:BabyBryant}) implies that
  \begin{equation} \label{eq:Define-V-in} \bry(k\sigma) = (k\sigma)^{-2} +
    \cO \left( (k\sigma)^{-4}\right), \qquad (\sigma\to\infty).
  \end{equation}
  By Theorem~\ref{thm:paraboliclayer-subsuper}, the subsolutions in the
  parabolic region $\Omega_\PAR $ also have limits at $t=0, \tau=-\infty$,
  namely
  \begin{equation} \label{eq:SmallTime2} V_\PAR^{-\infty}(\sigma) =
    \lim_{t\searrow0} v_\PAR^{-}(\sigma\theta, t)
    =\frac{(1-\gamma)(n-1)^2}{\sigma^2} - \frac{B^2}{\sigma^4}
  \end{equation}
  Equations~(\ref{eq:SmallTime1}) and (\ref{eq:SmallTime2}) show that to
  establish (\ref{eq:inner-parabolic-glue}), it will suffice to prove
  \begin{equation} \label{eq:vin-vpar-comparison} V_\IN^{-\infty}
    (\sigma_*) > V_\PAR^{-\infty} (\sigma_*) \text{ and } V_\IN^{-\infty}
    (3\sigma_*) < V_\PAR^{-\infty} (3\sigma_*)
  \end{equation}
  and then to choose $-\tau_*$ sufficiently large such that $V_\IN^{-}(\sigma_*,\tau)
  > V_\PAR^{-}(\sigma_*,\tau)$ and $V_\IN ^{-}(3\sigma_*,\tau) <
  V_\PAR^{-}(3\sigma_*,\tau)$ are preserved for all $\tau<\tau_*$.

  Without loss of generality, we may assume that $\frac12(n-1)^{-1}
  <k<2(n-1)^{-1}$.  Then by (\ref{eq:Define-V-in}), there is a constant
  $C_n<\infty$ such that
  \begin{equation}
    \left|\sigma^2 V_\IN^{-\infty} -k^{-2}\right|
    \leq C_n\sigma^{-2}
    \label{eq:k-bounds}
  \end{equation}
  for all $\sigma\geq1$ and all $k$ under consideration.

  Hence we have
  \[
  \sigma^2\left( V_\IN^{-\infty} - V_\PAR^{-\infty} \right) = k^{-2} -
  (1-\gamma)(n-1)^2 + (B^2+\vartheta C_n)\sigma^{-2},
  \]
  where $|\vartheta|\leq 1$.  In order to verify
  \eqref{eq:vin-vpar-comparison}, we must find the sign of the \textsc{lhs} for
  $\sigma=m\sigma_*$, with $m=1$ or $m=3$. Using $\sigma_* =
  B\varepsilon^{-1/2}$, we find that when $\sigma = m\sigma_*$,
  \[
  \sigma^2\bigl( V_\IN^{-\infty} - V_\PAR^{-\infty} \bigr) = k^{-2} -
  (1-\gamma)(n-1)^2 + \frac{\varepsilon}{m^2}(1+\vartheta C_nB^{-2}).
  \]
  In particular, \eqref{eq:vin-vpar-comparison} will hold if this quantity
  is positive for $m=1$ and negative for $m=3$.  We can achieve this by
  first choosing $B$ so large that we can ignore the term containing
  $\vartheta$: $B\geq B_n \doteqdot \sqrt{100C_n}$ will do.  Then we choose
  $k$ so as to satisfy
  \[
  k^{-2} - (1-\gamma)(n-1)^2 = -\frac\varepsilon2.
  \]
  Solving this for $k$ leads to the value $k_-(\delta, \varepsilon)$
  mentioned in the Lemma.  Once $k$ is given this value and $B$ is chosen
  large enough, \eqref{eq:vin-vpar-comparison} will hold.

\end{proof}

A similar statement holds for supersolutions.

\section{Subsequential convergence of the regularized solutions}
\label{sec:Compactness}

In this section, we prove compactness of the family
$\{g_\omega(t) \mid 0 < t \leq T\}$ for some $0<T<T_0$. We find a
convergent subsequence $g_{\omega_j}(t) \to g_*(t)$; we show that
$g_*(t)$ is a smooth solution on $\cS^{n+1}$ for small $t>0$; and
we verify that $g_*(t)$ is indeed a forward evolution from the
singular initial metric $g_0$.

We do this directly, rather than by invoking Hamilton's compactness theorem,
which instead gives $\eta_j^*(g_{\omega_j})(t) \to \tg_*(t)$ for some sequence
$\{\eta_j:\cS^{n+1}\to\cS^{n+1}\}$ of time-independent diffeomorphisms
fixing the north pole $\NP\in\cS^{n+1}$ \cite{HamCom}.

The main problem in establishing compactness is that, although we have
precise control of the $v$ function near the singular point, this
information is only valid in a neighborhood of the form $\{P\in\cS^{n+1}
\mid \psi(P, t) \leq \br\}$ for some $\br>0$.

\subsection{Splitting $\cS^{n+1}$ into regular and singular parts}
\label{sec:regular-singular-splitting}

The following lemma allows us to split the manifold $\cS^{n+1}$ into a
regular and a singular part.
\begin{lemma}[Existence of collars]\label{lem:collar}
  For any $0<m<1$ and small enough $\alpha>0$, there exists $T=T(\alpha)\in
  (0, T_0)$ such that for $0<t<T(\alpha)$, the function
  \[
  {v_\COL}^+(r, t) = \min\left\{ 1, me^{3t/\alpha^2} + \left(
  \frac{r-\br}{\alpha}\right)^2 \right\}
  \]
  is a supersolution of $v_t=\cF[v]$, while
  \[
  {v_\COL}^-(r, t) = \max\left\{ 0, me^{-t/\alpha^2} - \left(
  \frac{r-\br}{\alpha}\right)^2 \right\}
  \]
  is a subsolution.
\end{lemma}
\begin{proof}
  The statements are verified by direct computation.  For $G_\pm(z) = m\pm
  z^2$ one computes that
  \[
  \cF\left[ G_+\left( \frac{r-\br}{\alpha} \right) \right] = \alpha^{-2}
  \left[ 2m +\cO(\alpha) \right],
  \]
  and
  \[
  \cF\left[ G_-\left( \frac{r-\br}{\alpha} \right) \right] = \alpha^{-2}
  \left[ 2m-2\bigl(\frac{r-\br}{\alpha} \bigr)^2 +\cO(\alpha) \right],
  \]
  hold when $|r-\br|\leq \alpha$, and when $\alpha\leq \tfrac12 \br$.

  If we now let $m$ vary with time, we find that $m(t) - z^2$ is a
  subsolution for small $\alpha$ if $m'(t) < cm(t)/\alpha^2$ for any
  constant $c<0$: we choose $c=-1$ to get the subsolution in the Lemma.
  Similarly, $m(t)+z^2$ is a supersolution if $m'(t) > cm(t)/\alpha^2$ for
  any constant $c>2$; we choose $c=3$.
\end{proof}

In Theorem~\ref{pinch}, we saw that a smooth forward evolution (if one exists) must
lie between the sub- and super- solutions constructed in Section~\ref{sec:Informal}.
Now we show, using Lemma~\ref{lem:collar}, that the regularized solutions
$g_\omega(t)$ obey a similar bound, uniformly for small $\omega>0$.

\begin{lemma}
  Let $\varepsilon, \delta>0$ be given.  Then there exist $\br, \bt>0$ such
  that for each solution $g_\omega(t)$, with $\omega>0$ sufficiently small,
  one has
  \begin{equation}
    \label{eq:v-contained}
    v^-_{\varepsilon,\delta}(r, t+\omega)
    < v_\omega(r,t)
    < v^+_{\varepsilon,\delta}(r,t+\omega)
  \end{equation}
  for all $r\in(0,\br)$ and $t\in(0,\bt)$.
  \label{lem:v-contained}
\end{lemma}
\begin{proof}
  To apply our comparison principle, Lemma~\ref{yamm}, we need to show
  that the given solutions $v_\omega(r,t)$ cannot cross the barriers
  $v^\pm_{\varepsilon,\delta}(r,t)$ at the right edge $r=\br$ of their
  domain.  The barriers from Lemma~\ref{lem:collar} allow us to do this.

  We let $\br$ be as in the construction of the sub- and super- solutions
  $v^\pm_{\varepsilon,\delta}$.  Then we choose $m_\pm$ and $\alpha>0$ so
  that the sub- and super- solution from Lemma \ref{lem:collar} satisfy
  \[
  v_\COL^-(m_-, \alpha; r, 0) < \vini(r) < v_\COL^+(m_+,\alpha; r,0)
  \]
  for all $r$, and also so that
  \[
  v^-_{\varepsilon,\delta}(\br, 0) < v_\COL^-(m_-, \alpha; \br, 0) \text{
  and } v^+_{\varepsilon,\delta}(\br, 0) > v_\COL^+(m_+,\alpha; \br, 0).
  \]
  Finally, we choose $\bt>0$ so small that this last inequality persists for
  $0<t<\bt$.
\end{proof}

\subsection{Uniform curvature bounds for $g_\omega(t)$ for $t>0$}
\label{sec:uniform-curvature-bounds}

\begin{lemma}\label{lem:Riemann-bounded}
  The curvature of the regularized solutions $g_\omega(t)$ of Ricci flow constructed in
  Section~\ref{sec:initialData} is bounded uniformly by
  \begin{equation}
    \label{eq:Riemann-bounded}
    |\Rm| \leq C \frac{-\log t} {t}
  \end{equation}
  for $0<t<t_*$, where the constant $C$ depends only on the initial
  metric $g_0$ and chosen sub- and super- solutions $v^\pm$.
\end{lemma}

\begin{proof}
  For some small enough $\varepsilon,\delta>0$, we consider the sub- and
  super- solutions $v^\pm_{\varepsilon,\delta}$ constructed in
  Section~\ref{sec:Informal}, which are defined in the ``singular region''
  $\psi\leq r_*$.  Here $r_*=r_*(\varepsilon,\delta)\ll r_\#$.  We now
  choose $m^\pm$ and $\alpha<\tfrac12 r_*$ so small that
  \[
  v_\COL^+(r, t) = m^+e^{3t/\alpha^2}+(r-r_*)^2/\alpha^2 >\vini(r)
  \]
  and
  \[
  v_\COL^-(r, t) = m^-e^{-t/\alpha^2}-(r-r_*)^2/\alpha^2 <\vini(r).
  \]
  In the collar-shaped region $|\psi - r_*|\leq \frac12\alpha$, the quantity
  $v$ is now a bounded solution of $v_t=\cF[v]$ which is bounded away
  from zero.  Interior estimates for quasilinear parabolic \textsc{pde} imply that
  all higher derivatives will be bounded in an narrower collar $|\psi -
  r_*|\leq \frac13\alpha$.  In the ``regular region'' which begins with
  $\psi\geq r_*$ and extends to the south pole $\SP\in\cS^{n+1}$, we can
  now apply the maximum principle to
  \begin{equation}\label{eq:Rm-norm-growth}
    \frac{\partial |\Rm|^2} {\partial t}
    \leq \Delta |\Rm|^2 + c_n |\Rm|^3
  \end{equation}
  to conclude that $|\Rm|$ remains bounded for a short time there.  Since all
  higher derivatives of $\Rm$ can be expressed in terms of $v$ and its
  derivatives, we find that these too are bounded in the collar, and by
  inductively applying the maximum principle to equations for
  $\partial_t|\nabla^\ell\Rm|^2 - \Delta|\nabla^\ell\Rm|^2 $, one finds that
  higher derivatives of $\Rm$ also remain bounded in the regular region.

  It remains to prove that \eqref{eq:Riemann-bounded} holds in the singular
  region, namely the region where $\psi\leq r_*$.  We do this in the next
  Lemma.
\end{proof}
\begin{lemma}
  \label{lem:vrr-bounded}
  Let $0<v^-(r, t) < v^+(r,t)$ be sub- and super- solutions of $v_t =
  \cF[v]$ which we have constructed on the domain $0<r<r_*$, $0<t<t_*$, and
  suppose that $v$ is a solution of this equation which satisfies
  $v^-<v<v^+$.  Then there is a constant $C$ such that
  \[
  \left|\frac{1-v}{r^2}\right| + \left|\frac{v_r}{r}\right| \leq C
  \frac{-\log t}{t} \quad \text{ for }0<r<r_*\text{ and }0<t<t_*.
  \]
  The constant $C$ depends only on the sub- and super- solutions $v^\pm$
  and the constant $A$ from (M8) in Table~\ref{tab:initialdata}.
\end{lemma}
This Lemma implies that the sectional curvatures $L=(1-v)/r^2$ and
$K=-v_r/2r$ are uniformly bounded by $C(-\log t)/t$ for all solutions
caught between our sub- and super- solutions.  This completes the proof
of Lemma~\ref{lem:Riemann-bounded} above.

\medskip

Note that the Lemma predicts that $\sup |\Rc| \sim C(-\log t)/t$, which is
not integrable in time.  This is consistent with the fact that our initial
metric $g_0$ is not comparable to the standard metric on $S^{n+1}$. To wit,
there is no constant $c>0$ such that $cg_{\cS^{n+1}}\leq g_0\leq
\frac{1}{c}g_{\cS^{n+1}}$.

\begin{proof}
  The supersolution satisfies $v^+\leq 1$, so we immediately get $v\leq 1$.
  The subsolution is bounded from below by
  \[
  v^- \geq 1-C\sigma^2 = 1- C \bigl(\frac{-\log t}{t}\bigr)^2 r^2
  \]
  for some constant $C$.  This implies that
  \[
  \frac{1-v}{r^2} \leq \frac{1-v^-}{r^2} \leq C\frac{-\log t}{t}.
  \]
  To estimate the second derivative, we choose a small constant $\nu>0$ and
  split the rectangle $\Omega = (0, r_*) \times (0,t_*)$ into two pieces,
  \[
  \Omega_1 = \left\{ (r, t)\in\Omega : r^2 \leq \nu \frac{t}{-\log t}
  \right\} \text{ and } \Omega_2 = \left\{ (r, t)\in\Omega : r^2 \geq \nu
  \frac{t}{-\log t} \right\}.
  \]
  Throughout the whole region $\Omega$, the quantity $a = r^2(K-L)$ is
  bounded by (M8), Table~\ref{tab:initialdata}, so that we get
  \[
  \left|\frac{v_r}{2r}\right| = |K| \leq |L|+|a/r^2| \leq C\frac{-\log
  t}{t}+ \frac{A}{r^2}.
  \]
  In the region $\Omega_2$, the second term is dominated by the first, and we
  therefore find that $|v_r/r|\leq C(-\log t)/t$ on $\Omega_2$.

  In the first region $\Omega_1$, the equation $v_t=\cF[v]$ is uniformly
  parabolic, so that standard interior estimates for parabolic equations
  \cite[\S V.3, Theorem 3.1]{LSU67} provide an $(r, t)$-dependent bound for
  $v_r$.  The following scaling argument shows that this leads to the
  stated upper bound for $|v_r/r|$.

  Let $(\br, \bt) \in \Omega_1$ be given, and consider the rescaled
  function
  \[
  w(x, y) \doteqdot \frac{1-v(\br+\br x, \bt + \br^2 y)}{\beta}.
  \]
  Then $w$ satisfies
  \begin{equation}\label{eq:rescaled-pde-for-w}
    \frac{\partial w}{\partial y}
    =
    (1-\beta w)w_{xx}
    +\tfrac12 \beta w_x^2
    + \frac{n-2+\beta w}{1+x} w_x
    +  \frac{2(n-1)}{(1 + x)^2}  (1-\beta w) w.
  \end{equation}
  In the rectangle $-\tfrac12 < x < \tfrac12$, $-\tfrac12 < y < 0$, we have
  \[
  r = \br(1+x) \in \left( \tfrac12\br, \tfrac32\br \right),
  \]
  and
  \[
  \bt > t = \bt + \br^2 y > \bt - \tfrac12\br^2 \geq \left(
  1-\frac{C}{-\log t} \right)\bt >\tfrac12 \bt
  \]
  if $t_*$ is small enough.

  Hence $w$ will be bounded by
  \[
  0\leq w \leq \frac{C}{\beta} \frac{-\log \bt}{\bt} \br^2.
  \]
  We now choose $\beta = C \frac{-\log \bt}{\bt}\br^2$, which always
  satisfies $\beta\leq C\nu$, since $(\br, \bt)\in\Omega_1$.

  Our function $w$ is a solution of \eqref{eq:rescaled-pde-for-w} in which
  $1-\beta w = v\geq v^-$ is uniformly bounded from below, $1+x\geq
  \frac{1}{2}$ is bounded from below, and for which we have found $0\leq
  w\leq 1$.  Standard interior estimates now imply that $|w_x(0,0)| \leq C$
  for some universal constant.  After scaling back, we then find that
  \[
  |v_r(\br, \bt)| = \left|-\frac{\beta}{\br}w_x(0,0)\right| \leq
  C\frac{-\log\bt}{\bt}\br,
  \]
  as claimed.

\end{proof}

\subsection{Uniform curvature bounds at all times, away from the
singularity}
\label{sec:unif-cbounds-away-from-singularity}

The initial metric is smooth away from the singularity, and one expects
this regularity to persist for short time.  Here we prove this.

For $r>0$, we let $\cK_r$ be the complement in $\cS^{n+1}$ of the
neighborhood of the North Pole in which $\psi_\init<r$.
\begin{lemma}
  For any $r_1>0$ and $t_1>0$, there exists $C_1$ such that the
  Riemann tensor of any solution $g_\omega(t)$ is bounded by $|\Rm(x,
  t)|\leq C_1$ for all
  \[
  (x, t) \in (\cS^{n+1}\times[t_1, t_*]) \cup (\cK_{r_1}\times[0, t_*]).
\]
  Furthermore, all covariant derivatives of the curvature are
  uniformly bounded: for any $k<\infty$ there is $C_k$ such
  that $|\nabla^k\Rm| \leq C_k$ on $(\cS^{n+1}\times[t_1, t_*]) \cup
  (\cK_{r_1}\times[0, t_*])$ for all $\omega>0$.
\end{lemma}
\begin{proof}
  The curvature bound for $t\geq t_1$ was proved in
  Section~\ref{sec:uniform-curvature-bounds}. The derivative bounds
  for $t\ge t_1$ then follow from Shi's global estimates by taking
  $g(t_{1}/2)$ as the initial metric \cite{Shi89}.

  Now we estimate $\Rm$ in $\cK_{r_1}$ for $0<t<t_1$,
  for sufficiently small $t_1$. On $\partial\cK_{r_1}$ we have
  $\psi_\init = r_1$, and thus, for small enough $\omega>0$,
  $\psi_\omega(x, 0) = \psi_\init =r_1$.  By
  \eqref{eq:ddt-psi-squared-bound}, we therefore have
  \[
  \tfrac12 r_1^2 \leq \psi_\omega^2 \leq \tfrac32 r_1^2\quad \text{ for }\quad 0\leq
  t\leq \frac{r_1^2}{4(\amax+n)}.
  \]
  We choose $t_1< r_1^2/4(\amax+n)$ and estimate $|\Rm|$ on
  $\partial\cK_{r_1}$ using the $v = \psi_s^2$ function.  For any $r_1>0$,
  the functions $v_\omega(r,t)$ corresponding to the metrics $g_\omega(t)$
  will be bounded away from zero on the interval $\tfrac12 r_1 \leq r \leq
  \tfrac32 r_1$ and for $0\leq t\leq t_1$, if $t_1$ is small enough.  Being
  solutions of the nondegenerate quasilinear parabolic equation
  $v_t=\cF[v]$, all derivatives of the $v_\omega$ are uniformly bounded in
  the smaller region $\sqrt{1/2}r_1 < r < \sqrt{3/2}r_1$. Hence $K$, $L$
  and $|\Rm|$ are bounded on $\partial\cK_{r_1}$ for $0\leq t \leq t_1$.
  Applying the maximum principle to equation \eqref{eq:Rm-norm-growth},
  and possibly reducing $t_1$ one last time, we find that $|\Rm|$ is indeed
  bounded on $\cK_{r_1} \times [0,t_1]$.

  Observe that all derivatives $\nabla^k\Rm$ are initially bounded, and
  also that they are bounded in the collar region, because they can be
  expressed there in terms of   $v_\omega(r, t)$ and its derivatives.
  Thus the derivative estimates follow from a modification of Shi's local
  estimates, which allow one to take advantage of bounds on the curvature
  of the initial metric \cite[Theorem~14.16] {YayAnotherBook}.
\end{proof}

\subsection{Constructing the solution}
\label{sec:constructing-the-solution}

Now we show that a subsequence of our regularized solutions converges to the
solution described in our main theorem, i.e.~to a
smooth forward evolution from a singular initial metric $g_0$. We do this in
two steps. We first demonstrate subsequential convergence to a limit solution that
has all the properties we want, except possibly at the north pole $\NP$. Then we
find coordinates in which this limit is smooth at $\NP$ for all $t>0$ that it
exists. Together, these results complete the proof of Theorem~\ref{thm:main}.

\begin{lemma}
  Any sequence $\omega_i\searrow0$ has a subsequence
  $\omega_{i_k}$ for which the solutions $g_{\omega_{i_k}}(t)$
  converge in $C^\infty_{\rm loc}$
  on $\bigl(\cS^{n+1} \setminus \{\NP\}\bigr) \times [0, t_*]$.
  \label{lem:convergence-away-from-NP}
\end{lemma}

\begin{proof}
  Let $\omega_i\searrow0$ be any given sequence.  For arbitrary $r_1>0$, we
  have just shown that the metrics $g_{\omega_i}(t)$ have uniformly bounded
  Riemann curvature tensors on $\cK_{r_1/2} \times [0, t_*]$.  Hence there
  exists a subsequence $\omega_{i_k}$ such that the solutions
  $g_{\omega_{i_k}}(\cdot)$ converge in $C^\infty$ on $\cK_{r_1} \times [0, t_*]$.
  For proofs, see Lemma~2.4 and Hamilton's accompanying arguments in \cite{HamCom},
  or else Lemma 3.11 and the subsequent discussion in \cite[\S3.2.1] {YetAnotherBook}.

  A diagonalization argument then provides a further subsequence which
  converges in $C^{\infty}_{\rm loc}$ on
  $\bigl(\cS^{n+1} \setminus \{\NP\}\bigr) \times [0, t_*]$.
\end{proof}

To complete the construction, we show in the next Lemma that the apparent
  singularity at the north pole is removable.

\begin{lemma}
  Let $g_*(t)$ be the limit of any convergent sequence
  $g_{\omega_i}(t)$.  Then there is a homeomorphism
  $\Phi:\cS^{n+1} \to \cS^{n+1}$ for which the metrics $\tg(t)
  \doteqdot \Phi^*[g_*(t)]$ extend to a smooth metric on
  $\cS^{n+1}$ for all $t>0$ that they exist.

  The homeomorphism is smooth except at the north pole.
  \label{lem:singularity-off}
\end{lemma}

Because $g_*(t)$ satisfies Ricci flow, the modified family of metrics
$\tg(t)$ also satisfies Ricci flow on the punctured sphere.  But since
the modified metrics extend to smooth metrics on the whole sphere, they
constitute a solution of Ricci flow that is defined everywhere on
$\cS^{n+1} \times [0, t_*]$, except at the north pole $\NP$ at time $t=0$.
Its initial value is $\Phi^*[g_0]$, so it is the solution we seek.

\begin{proof}
  Choose any time $t_1\in (0, t_*)$, and let the homeomorphism
  $\Phi$ be such that
  \[
  \tg(t_1)  = (dx)^2 + \tpsi(x, t_1)^2\gcan.
  \]
  Then, since $v(r, t_1)$ is a smooth function, one finds that
  $\tg(t_1)$ extends to a smooth metric at the north pole.

  We write $\nabla$ for the Levi-Civita connection of $\tg(t)$ at any given time
  $t\in (0, t_*)$, and $\bar\nabla$ for the connection at the fixed time $t_1$.

  We will show that the metrics $\tg(t)$ extend smoothly across the
  north pole by finding uniform bounds for the derivatives
  $\bar\nabla^k\tg(t)$.  These bounds imply that for any pair of
  smooth vector fields $V, W$ defined near the north pole, the
  function $\tg(t)(V, W)$ and its derivatives are uniformly bounded
  near the north pole, so that $\tg(t)(V, W)$ extends smoothly.
  The metric $\tg(t)$ therefore also extends smoothly.

  To estimate the derivatives $\bar\nabla^k\tg(t)$, we start with the
  identity $\nabla^k\tg(t) \equiv 0$ and then estimate the difference
  between the connections $\bar\nabla$ and $\nabla$.  This difference
  is, as always, a tensor field, so that we may write
  \[
  \nabla = \bar\nabla + A, \qquad
  A = A^i_{jk} \frac{\partial}{\partial x^i}\otimes dx^j\otimes dx^k.
  \]
  The tensor $A$ vanishes at time $t=t_1$, and it evolves by
  \begin{equation}
    \partial_t A^i_{jk} =
    -g^{i\ell}\bigl\{
    \nabla_j R_{\ell k} + \nabla_k R_{j \ell} - \nabla_\ell R_{jk}
    \bigr\}.
    \label{eq:A-evolution}
  \end{equation}
  Our construction of $\tg(t)$ was such that all covariant derivatives
  of the Riemann curvature of $\tg(t)$ are bounded on any time
  interval $\varepsilon\le t\le t_*$.  In particular, the Ricci tensor
  is uniformly bounded, and hence the metrics $\tg(t)$ are equivalent,
  in the sense that
  \begin{equation}
    \frac{1}{C_1(t)} \tg(t_1) \leq \tg(t) \leq C_1(t) \tg(t_1)
    \label{eq:metrics-equivalent}
  \end{equation}
  for some constants $C_1(t)$ (which become unbounded as $t\searrow0$).
  Because of this, uniform bounds for some tensor with respect to one
  metric $\tg(t)$ are equivalent to uniform bounds with respect to any
  other metric $\tg(t')$, as long as $t,t'>0$.

  Since $\nabla_i R_{jk}$ is bounded, \eqref{eq:metrics-equivalent} and
  \eqref{eq:A-evolution} imply that $\partial_tA$ is uniformly
  bounded, so that we have
  \begin{equation}
    |A|\leq C(t)|t-t_1|,
    \label{eq:C0-A-estimate}
  \end{equation}
  in which the constants $C(t)$ again deteriorate as $t\searrow0$.

  Because $\bar\nabla\tg(t) = (\nabla -A)\tg(t) = -A*\tg(t)$ implies that
  $\bar\nabla\tg(t)$ is bounded, this leads to a Lipschitz estimate
  for the metric $\tg(t)$.

  Next, we estimate $\bar\nabla A$ by applying the time-independent
  connection $\bar\nabla$ to both sides of \eqref{eq:A-evolution}.
  This tells us that
  \[
  \partial_t \bar\nabla A = \bar\nabla\partial_t A
  = \nabla(\partial_t A) - A*A.
  \]
  We have just bounded $A$, so the second term is bounded, while the
  first term must also be bounded since it is given by
  \eqref{eq:A-evolution}.  Thus we find that
  \begin{equation}
    |\bar\nabla A| \leq C(t) |t-t_1|
    \text{ and }
    |\nabla A| \leq C(t) |t-t_1|.
    \label{eq:C1-A-estimate}
  \end{equation}
  Inductively, one bounds all higher derivatives, $\bar\nabla^k A$, and
  $\nabla^k A$.

  To bound $\bar\nabla^k F$ for any tensor $F$ we write this derivative
  as $\bar\nabla^kF = (\nabla -A)^k F$.  Expanding this leads to a
  polynomial in  $F, \nabla F, \cdots, \nabla^k F$ and $A, \nabla A,
  \cdots, \nabla^{k-1}A$.  The bounds for $\nabla^jA$ which we have
  just established then show that bounds for $F, \nabla F, \cdots,
  \nabla^k F$ imply similar bounds for $F, \bar\nabla F, \cdots,
  \bar\nabla ^k F$.

  In the special case where $F=\tg(t)$, we have the trivial bounds
  $\nabla^k \tg(t) = 0$, so that all derivatives $\bar\nabla^k\tg(t)$
  are bounded.  As stated before, this implies that the metrics
  $\tg(t)$  extend smoothly to $\cS^{n+1}$.
\end{proof}

\appendix

\section{Rotationally symmetric neckpinches}
\label{sec:DifferentiateTheAsymptotics}

Here we recall relevant notation and results from \cite{AK04, AK07}.
Remove the poles $P_{\pm}$ from $\cS^{n+1}$ and identify $\cS
^{n+1}\backslash\{P_{\pm}\}$ with $(-1,1)\times\cS^{n}$.  In \cite{AK04},
we considered Ricci flow solutions whose initial data are smooth
$\mathrm{SO}(n+1)$-invariant metrics of the form
\begin{equation}
  g=\varphi(x)^2\,(dx)^2+\psi(x)^2\,\gcan ,
  \label{eq:InitialDatum}
\end{equation}
where $x\in(-1,1)$.  Parameterizing by arc length $s(x)\doteqdot\int_0
^{x}\varphi(y)\,dy$ and abusing notation, we wrote (\ref{eq:InitialDatum})
in a more geometrically natural form,\footnote{The choice of $s$ as a
parameter has the effect of fixing a gauge, thereby making Ricci flow
strictly parabolic.} namely
\[
g=(ds)^2+\psi(s)^2\,\gcan .
\]
Smoothness at the poles is ensured by the boundary conditions that
$\lim_{x\to\pm1}\psi_{s}=\mp1$ and that $\psi/(s_{\pm}-s)$ be a smooth even
function of $s_{\pm}-s$, where $s_{\pm}\doteqdot s(\pm1)$.

Metrics of the form (\ref{eq:InitialDatum}) have two distinguished
sectional curvatures: let $K$ denote the curvature of the $n$ $2$-planes
perpendicular to a sphere $\{x\}\times\cS^{n}$, and let $L$ denote the
curvature of the $\binom{n}2$ $2$-planes tangential to
$\{x\}\times\cS^{n}$.  These sectional curvatures are given by the formulas
\begin{equation}
  K\doteqdot-\frac{\psi_{ss}}{\psi}\qquad\text{and}\qquad L\doteqdot\frac
  {1-\psi_{s}^2}{\psi^2}.  \label{eq:Sectionals}
\end{equation}
Evolution of the metric (\ref{eq:InitialDatum}) by Ricci flow is equivalent
to the coupled system of equations
\begin{equation}
  \left\{
  \begin{aligned}
    \psi_t & =\psi_{ss}-(n-1)\frac{1-\psi_s^2}{\psi} =-\psi[ K+(n-1)L]\\
    \varphi_t & =\frac{n\psi_{ss}}{\psi}\varphi = -n K\varphi
  \end{aligned}\right.
  \label{eq:RF-psi}
\end{equation}
in which $\partial_s$ is to thought of as an abbreviation of $\varphi(x,
t)^{-1}\partial_x$.

In \cite{AK04, AK07}, we called local minima of $\psi$ \emph{necks} and
local maxima \emph{bumps}.  We called the region between a pole and its
closest bump a \emph{polar cap}.  In \cite{AK04}, neckpinch singularity
formation was established for an open set of initial data of the form
(\ref{eq:InitialDatum}) on $\cS^{n+1}$ $(n\geq2)$ satisfying the following
assumptions: (1) The metric has at least one neck and is ``sufficiently
pinched'', i.e.~the value of $\psi$ at the smallest neck is sufficiently
small relative to its value at either adjacent bump.  (2) The sectional
curvature $L$ of planes tangent to each sphere $\{x\}\times \cS^{n}$ is
positive.  (3) The Ricci curvature $\Rc =nK\,ds^2+[K+(n-1)L]\psi^2\,\gcan $
is positive on each polar cap.  (4) The scalar curvature $R=2nK+n(n-1)L$ is
positive everywhere.  In \cite{AK07}, precise asymptotics were derived
under the additional hypothesis: (5) The metric is reflection symmetric,
i.e.~$\psi(s)=\psi(-s)$, and the smallest neck is at $x=0$.

To describe the asymptotic profile derived for such data, let $T_0<\infty$
denote the singularity time, let $\Psi(\sigma,\tau)=\psi(s,t)/\sqrt
{2(n-1)(T_0-t)}$ be the ``blown-up'' radius, where $\sigma=s/\sqrt{T_0-t} $
is the rescaled distance to the neck and $\tau=-\log(T_0-t)$ is rescaled
time.  Then the main results of \cite{AK07} can be summarized as follows.

\begin{theorem}
  \label{Theorem-AK2}For an open set of initial metrics satisfying
  assumptions~(1)--(5) above, the solution $(\cS^{n+1},g_{t})$ of Ricci
  flow becomes singular at a time $T_0<\infty$ depending continuously on
  $g_0$.  The diameter remains uniformly bounded for all $t\in[0,T_0]$.
  The metric becomes singular only on the hypersurface $\left\{ 0\right\}
  \times\cS^{n}$.  The solution satisfies the following asymptotic profile.

  \textbf{Inner region: }on any interval $|\sigma|\leq A$, one has
  \[
  \Psi(\sigma,\tau)=1+\frac{\sigma^2-2}{8\tau}+o(\frac1{\tau})\qquad
  \text{uniformly as }\tau\to\infty.
  \]

  \textbf{Intermediate region: }on any interval $A\leq|\sigma|\leq
  B\sqrt{\tau} $, one has
  \[
  \Psi(\sigma,\tau)=\sqrt{1+\left( 1+o(1)\right) \frac{\sigma^2}{4\tau}
  }\qquad\text{uniformly as }\tau\to\infty.
  \]

  \textbf{Outer region: }for $0<\left| s\right| \ll1$, there exists a
  function $h$ such that
  \[
  \psi(s,t)=\left[ 1+h(\left| s\right| ,T_0-t)\right] \frac
  {\sqrt{n-1}}2\frac{\left| s\right| }{\sqrt{-\log|s| }},
  \]
  where $h(a,b)\to0$ as $a+b\searrow0$.
\end{theorem}

To flow forward from a neckpinch singularity, one must show that these
asymptotics can be differentiated.

\begin{lemma}
  \label{cor:AK2-Differentiable} Under the hypotheses of
  Theorem~\ref{Theorem-AK2}, the limit
  \[
  \psi(s)=\lim_{t\nearrow T_0}\psi(s(x,t),t)
  \]
  exists and satisfies
  \begin{equation}
    \psi(s)=\left[  1+h_0(s)\right]  \frac{\sqrt{n-1}}2\frac{s}{\sqrt{-\log
    s}}.  \label{initial-psi}
  \end{equation}
  This asymptotic profile may be differentiated.  In particular, one has
  \begin{equation}
    \psi_{s}(s)=\left[  1+h_1(s)\right]  \frac{\sqrt{n-1}}2\frac1
    {\sqrt{-\log s}} \label{initial-psi_s}
  \end{equation}
  for $s>0$, where $h_0(s)\to0$ and $h_1(s)\to0$ as $s\searrow0$.
\end{lemma}

\begin{proof}
  Existence of the limit and equation~(\ref{initial-psi}) follow directly
  from the results in Section~2.19 of \cite{AK07}.

  Recall that \cite{AK07} proves there exist $T_*<T_0$ and $x^*>0 $ such
  that $\psi_{s}(s(x,t),t)\geq0$ and $\psi_{ss}(s(x,t),t)\geq0$ in the
  space-time region $[0,x^*]\times[ T_*,T_0)$.  Moreover, the limit
  $s(x,T_0)$ exists for all $x\in[0, x^*]$.  Fix any $\underline{x}
  ,\overline{x}\in[0,x^*]$ with $\underline{x}<\overline{x}$.  Then one has
  \[
  \frac{\partial}{\partial t}\left\{ s(\overline{x},t)-s(\underline{x}
  ,t)\right\}
  =-n\int_{s(\underline{x},t)}^{s(\overline{x},t)}K(s(x,t),t)\,
  ds\,=n\int_{s(\underline{x},t)}^{s(\overline{x},t)}\frac{\psi_{ss}
  (s(x,t),t)}{\psi(s(x,t),t)}\,ds\,\geq0
  \]
  for all $t\in[ T_*,T_0)$, and
  \[
  \psi(s(\overline{x},t),t)-\psi(s(\underline{x},t),t)=\int_{s(\underline{x}
  ,t)}^{s(\overline{x},t)}\psi_{s}(s(x,t),t)\,ds\,\geq\int_{s( \underline
  {x},T_*)}^{s(\overline{x},T_*)}\psi_{s}(s(x,t),t)\, ds\,\geq0.
  \]
  Letting $t\nearrow T_0$, one sees that $\psi(s)\equiv\psi(s,T_0)$ is
  monotone increasing in small $s>0$, hence is differentiable almost
  everywhere.

  In what follows, $h_{i}(s)$ denotes a family of functions with the
  property that $h_{i}(s)\to0$ as $s\searrow0$.  First observe that
  \begin{equation}
    \left\{
    \begin{aligned}
      \int_0^{\hat{s}}(-\log s)^{-1/2}\,d s\, & =\frac{\hat{s}}
      {\sqrt{-\log\hat{s}}}
      -\frac12\int_0^{\hat{s}}(-\log s)^{-3/2} \,d s\,\\
      & =\left[ 1+h_2(\hat{s})\right]
      \frac{\hat{s}}{\sqrt{-\log\hat{s}}}.
    \end{aligned}
    \right.
    \label{eq:IntegrationEstimate}
  \end{equation}
  Now suppose there exists a fixed $\varepsilon>0$ such that
  \[
  \psi_{s}(s)\geq\left( 1+\varepsilon\right) \frac{\sqrt{n-1}}2\frac
  1{\sqrt{-\log s}}
  \]
  for a.e.~small $s>0$.  Then applying (\ref{initial-psi}) and
  (\ref{eq:IntegrationEstimate}), one obtains
  \begin{align*}
    \psi(\hat{s})=\int_0^{\hat{s}}\psi_{s}(s)\,ds\, & \geq\left(
    1+\varepsilon\right) \left[ 1+h_2(\hat{s})\right] \frac{\sqrt{n-1}}
    2\frac{\hat{s}}{\sqrt{-\log\hat{s}}}\\
    & \geq\left( 1+\varepsilon\right) \left[ 1+h_3(\hat{s})\right]
    \psi\left( \hat{s}\right) .
  \end{align*}
  Dividing by $\psi(\hat{s})>0$ yields $1\geq\left( 1+\varepsilon\right)
  \left[ 1+h_3(\hat{s})\right] $, which is impossible for small $\hat{s}$.
  It follows that
  \[
  \psi_{s}(s)\leq\left[ 1+h_4(s)\right] \frac{\sqrt{n-1}}2\frac1
  {\sqrt{-\log s}}.
  \]
  The complementary inequality is proved similarly.
\end{proof}

\section{Expanding Ricci solitons}

\label{app:AntiParabolic}

In this appendix, we prove a pair of lemmas which establish the assertion,
made in Section~\ref{sec:Formal}, that every complete expanding soliton
corresponding to a solution of
\begin{equation}
  U^2U_{\rho\rho}+\left\{  \frac{n-1-U^2}{\rho}+\frac{\rho}2\right\}
  U_{\rho}+\frac{n-1}{\rho^2}\left(  1-U^2\right)  U=0
\end{equation}
emerges from conical initial data,
\[
g=\frac{(dr)^2}{U_{\infty}^2}+r^2\gcan ,
\]
with $U_{\infty}>0$.

\begin{lemma}
  \label{UstaysPositive}Let $U(\rho)$ be a solution of
  \eqref{eq:expander-ode} on an interval $a\leq\rho<A\leq\infty$.

  If $U'<0$ and $0<U<1$, then $\lim_{\rho\to A}U(\rho)>0$.
\end{lemma}

\begin{proof}
  Since $U$ is decreasing, we may assume that $\lim_{\rho\to A}U=0$ in
  order to reach a contradiction.  Assuming, as we may, that $a$ is
  sufficiently large, we have
  \begin{align*}
    -\frac{U'}{U} & \leq-\left[ \tfrac12\rho-\frac{U^2}{\rho}+
    \frac{n-1}{\rho}\right]  \frac{U'}{U}\\
    &  =UU''+\frac{n-1}{\rho^2}(1-U^2)\\
    & \leq UU''+\frac{n-1}{\rho^2}
  \end{align*}
  for all $\rho\geq a$.  Integrate from $a$ to $b\in(a,A)$ to obtain
  \begin{align*}
    \log U(a)-\log U(b)  &  =\int_{a}^{b}\frac{U'}{U}\,d \rho\,\\
    & \leq\int_{a}^{b}\left[ UU''+\frac{n-1}{\rho^2}\right] \, d
    \rho\,\\
    &=[UU']_{a}^{b}
    -\int_{a}^{b}(U')^2d\rho+\int_{a}^{b}\frac{n-1}{\rho^2}\,d \rho\,\\
    & \leq U(b)U'(b)-U(a)U'(a)+\int_{a}^{b}\frac{n-1}{\rho^2
    }\,d \rho\,\\
    & \leq-U(a)U'(a)+\frac{n-1}{a}.
  \end{align*}
  Keeping $a$ fixed, we see that $\log U(b)$ remains bounded from above for
  all $b\in(a,A)$, and hence that $U(b)$ is bounded away from zero as
  $b\nearrow A$.
\end{proof}

\begin{lemma}
  \label{ConesAreAllWeGot}Let $U(\rho)$ be a maximal solution of
  \eqref{eq:expander-ode} defined for $0<\rho<A$.  Assume that
  $0<U(\rho)<1$ for all $\rho\in(0,A)$ and that
  \[
  \lim_{\rho\searrow0}U(\rho)=1.
  \]
  Then $A=\infty$, while $U$ is strictly decreasing and is bounded from
  below by the positive quantity $U_{\infty}=\lim_{\rho\to\infty}U(\rho)$.
\end{lemma}

\begin{proof}
  At any point where $U'=0$, the differential equation
  (\ref{eq:expander-ode}) forces $U''<0$.  Thus every critical point of $U$
  is a local maximum.  Since the solution stays between $0$ and $1$ and
  starts at $U(0)=1$, it can have no critical points for $0<\rho<A$, hence
  must be decreasing.

  If the maximal solution exists, then either $A<\infty$ and
  $\lim_{\rho\nearrow A}U(\rho)=0$, or else $A=\infty$.
  Lemma~\ref{UstaysPositive} rules out the possibility that $U$ reaches $0$
  at some finite value of $\rho$, so we must have $A=\infty$.

  But when $A=\infty$, Lemma~\ref{UstaysPositive} still applies and
  guarantees that $U(\infty)>0$.
\end{proof}

\section{The Bryant steady soliton}
\label{sec:Bryant}

A time-independent solution of the \textsc{pde}
(\ref{eq:ExactEquationInnerRegion}) satisfied by a solution $V(\sigma,t)$
in the inner region is a solution of the \textsc{ode} $\cF [V]=0$.  In this
appendix, we find and describe all complete time-independent solutions.

It is convenient here to consider $U=\sqrt{V}$.  It follows from
(\ref{eq:u-pde}) that a stationary solution $U$ must satisfy
\begin{equation}
  \frac{d}{d\sigma}\left[  \frac{dU}{d\sigma}-\frac{n-1}{\sigma}\bigl(U^{-1}
  -U\bigr)\right]  -\frac{n}{\sigma}\frac{dU}{d\sigma}=0.
  \label{eq:U_stationary}
\end{equation}
We introduce a new coordinate
\begin{equation}
  \zeta\doteqdot\log\sigma, \label{eq:Define-zeta}
\end{equation}
in terms of which the metric can be written as
\begin{equation}
  g=\frac{(d\sigma)^2}{U^2}+\sigma^2\gcan =e^{2\zeta}
  \bigl[\frac{(d\zeta)^2}{U^2}+\gcan \bigr].  \label{eq:OurMetric}
\end{equation}
If one writes equation (\ref{eq:U_stationary}) in terms of this new
coordinate, and defines
\[
\Upsilon\doteqdot U_{\zeta}-(n-1)(U^{-1}-U),
\]
one sees that solutions of (\ref{eq:U_stationary}) correspond to solutions
of the autonomous \textsc{ode} system
\begin{equation}
  \frac{d\Upsilon}{d\zeta}=(n+1)\Upsilon+n(n-1)(U^{-1}-U),\qquad\frac{dU}
  {d\zeta}=\Upsilon+(n-1)(U^{-1}-U).  \label{eq:ODESystem}
\end{equation}
Observe that $(\Upsilon,U)=(0,1)$ is a saddle point of this system.  In
fact, the linearization at $(\Upsilon,U)=(0,1)$ is the system
\[
\begin{pmatrix}
  \tilde{\Upsilon}'\\
  \tilde{U}'
\end{pmatrix}
=
\begin{pmatrix}
  n+1 & -2n(n-1)\\
  1 & -2(n-1)
\end{pmatrix}
\begin{pmatrix}
  \tilde{\Upsilon}\\
  \tilde{U}
\end{pmatrix}
\]
with eigenvalues $-(n-1)$ and $2$.  The unstable manifold $W^{u}$ of this
fixed point consists of two orbits.  One orbit lies in the region
$U>1,\Upsilon>0$.  As $\zeta\to\infty$ this orbit becomes unbounded, and
one finds that $U\to\infty$.  One finds that it does not generate a
complete metric on $\R^{n+1}$.  The other orbit lies in the region
$\{(\Upsilon ,U):0<U<1,\Upsilon<0\}$.

\begin{figure}[th]
  \centering
  \includegraphics[scale=0.8]{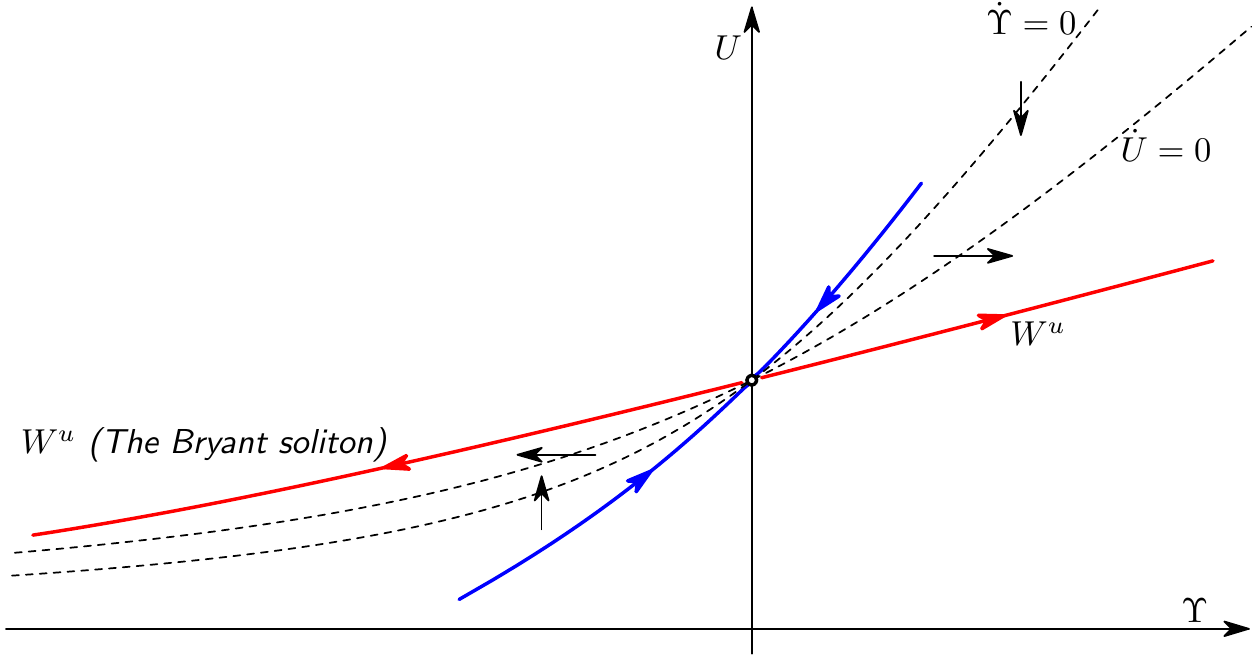}\caption{The $(\Upsilon, U)$
  phase plane ($n=2$)}
  \label{fig:U-Ups-phaseplane}
\end{figure}

\begin{lemma}
  \label{BryantSteadyState} The Bryant steady soliton is, up to scaling,
  the unique complete solution of (\ref{eq:U_stationary}) satisfying
  $0<U<1$.
\end{lemma}

\begin{proof}
  It is well known that the Bryant soliton is the unique (up to scaling)
  complete, rotationally symmetric steady gradient soliton on $\R^{n+1} $,
  for all $n\geq2$.  The point of the lemma is to exhibit it as a solution
  of (\ref{eq:U_stationary}), i.e.~an unstable trajectory of
  (\ref{eq:ODESystem}) emerging from the fixed point $(\Upsilon, U)=(0,1)
  $.

  We write the Bryant soliton $(\R^{n+1},\bar{g},\mathrm{grad}f)$ as it
  appears in \cite[Chapter~1, Section~4]{YetAnotherBook}, following
  unpublished work of Robert Bryant.  The metric $\bar{g}$ is defined in
  polar coordinates on
  $\R^{n+1}\setminus\{0\}\approx(0,\infty)\times\cS^{n}$ by
  \begin{equation}
    \bar{g}=ds^2+w^2(s)\,\gcan .  \label{eq:BryantMetric}
  \end{equation}
  The soliton flows along the vector field $\mathrm{grad}f$, where $f$ is
  the soliton potential function, i.e.~the solution of $\Rc(\bar{g}
  )+\bar{\nabla}\bar{\nabla}f=0$.  For this to hold, it is necessary and
  sufficient that $w(s)$ and $f(s)$ satisfy the system
  \begin{equation}
    w_{ss}=w_{s}f_{s}+(n-1)\frac{1-w_{s}^2}{w},\qquad f_{ss}=n\frac{w_{ss}}{w}.
    \label{eq:BryantSystem}
  \end{equation}

  Because $U>0$ for solutions of interest, we substitute $\sigma=w(s)$ and
  $U=w_{s}(s)$, thereby transforming (\ref{eq:OurMetric}) into
  (\ref{eq:BryantMetric}).  Then recalling from (\ref{eq:Define-zeta}) that
  \[
  d\zeta=\,\frac{d\sigma}{\sigma}=\frac{w_{s}}{w}\,ds,
  \]
  one sees that (\ref{eq:ODESystem}) becomes
  \begin{equation}
    \label{eq:OurSystemRedux}w_{ss} =w_{s}\frac{\Upsilon}{w}+(n-1)\frac
    {1-w_{s}^2}{w},\qquad\Upsilon_{s} =w_{s}f_{s}+nw_{ss}.
  \end{equation}
  The choice $\Upsilon=wf_{s}$ transforms (\ref{eq:OurSystemRedux}) into
  (\ref{eq:BryantSystem}).

  In \cite{YetAnotherBook}, following Bryant's work, the Bryant soliton is
  recovered from a careful analysis of trajectories of the \textsc{ode}
  system \cite[equation~(1.48)]{YetAnotherBook} corresponding to $x=U$ and
  $y=nU-\Upsilon$ near the saddle point $(x,y)=(1,n)$ corresponding to
  $(U,\Upsilon)=(1,0)$.  There it is shown that there exists a unique
  unstable trajectory (corresponding to $U\to0$ and $\Upsilon\to-\infty$ as
  $s$ increases) that results in a (non-flat) complete steady gradient
  soliton.  Moreover, this solution is unique up to rescaling.
\end{proof}

Let $U:(0,\infty)\to(0,1)$ be a solution of (\ref{eq:U_stationary}) which
corresponds to the Bryant steady soliton.  Every other such solution of
(\ref{eq:U_stationary}) is given by $U(k\sigma)$ for some $k>0$.

We now note a few simple facts about the behavior of the Bryant steady
soliton.  Its properties near infinity are well known, but its properties
near the origin are not as readily found in the literature.

\begin{lemma} \label{lem:BabyBryant}$~$

  \begin{enumerate}
  \item $U $ is strictly monotone decreasing for all $\sigma>0$.

  \item Near $\sigma=0$, $U(\sigma)^2$ is a smooth function of $\sigma^2 $,
    with the asymptotic expansion
    \[
    U(\sigma)^2=1+b_2\sigma^2+\frac{n}{n+3}b_2^2\sigma^4+\frac
    {n(n-1)}{(n+3)(n+5)}b_2^3\sigma^6+\cdots,
    \]
    where $b_2<0$ is arbitrary.

  \item Near $\sigma=+\infty$, $U^2$ has the asymptotic expansion
    \[
    U (\sigma)^2=c_2\sigma^{-2}+\frac{4-n}{n-1}c_2^2\sigma^{-4}
    +\frac{(n-4)(n-7)}{(n-1)^2}c_2^3\sigma^{-6}+\cdots,
    \]
    where $c_2>0$ is arbitrary.
  \end{enumerate}
\end{lemma}

\begin{proof}
  \textbf{(1)} It is well known (see
  e.g.~\cite[Lemma~1.37]{YetAnotherBook}) that the sectional curvatures of
  the Bryant soliton are strictly positive for all $\sigma>0$ and have the
  same positive limit at the origin.  By (\ref{eq:Sectionals}), the
  sectional curvature $K$ of a plane perpendicular to the sphere
  $\{\sigma\}\times\cS^{n}$ and the sectional curvature $L $ of a plane
  tangent to the sphere $\{\sigma\}\times\cS^{n}$ are
  \[
  K=-\frac{UU_{\sigma}}{\sigma}\quad\text{and}\quad L=\frac{1-U^2}{\sigma^2
  },
  \]
  respectively.  Hence $U'(\sigma)<0$ and $U(\sigma)<1$ for all $\sigma>0$.

  \textbf{(2) }We showed in Lemma~\ref{BryantSteadyState} that $U=w'$.
  Because $w$ is an odd function that is smooth at zero \cite{Ivey92}, $U$
  is an even function that is smooth at zero.  Thus $U^2$ has an asymptotic
  expansion near $\sigma=0$ of the form
  \[
  U(\sigma)^2=1+b_2\sigma^2+b_4\sigma^4+b_6\sigma^6+\cdots.
  \]
  By part (1), we must have $b_2<0$.  (Arbitrariness of $b_2$ corresponds
  to the invariance of $U$ under the rescaling $\sigma\mapsto k\sigma$.)
  The remaining coefficients are determined by the equation $\cF_\IN
  [U^2]=0$, which implies that
  \[
  0=\left[ 2(n+3)b_4-2nb_2^2\right] \sigma^2+4\left[ (n+5)b_6
  -(n-1)b_2b_4\right] \sigma^4+\cdots.
  \]

  In particular,
  \[
  b_4=\frac{n}{n+3}b_2^2>0\qquad\text{and}\qquad b_6=\frac{n-1}{n+5}
  b_2b_4<0.
  \]
  The claimed expansion follows.

  \textbf{(3)} Let $\xi=\sigma^{-1}$.  Then $U $ has an asymptotic
  expansion near $\xi=0$ of the form
  \[
  U(\sigma)^2=c_0+c_2\xi^2+c_4\xi^4+c_6\xi^6+\cdots.
  \]
  It is well known (see e.g.~\cite[Remark~1.36]{YetAnotherBook}) that
  $C^{-1}\sqrt{s}\leq w(s)\leq C\sqrt{s}$ for $s$ bounded away from $0$.
  This forces $c_0=0$.  One finds that $\cF [U^2] =0$ if and only if
  \[
  \xi^2U^2U_{\xi\xi}^2 -\frac12(\xi U_{\xi}^2)^2 -(n-1-3U^2)\xi U_{\xi}^2
  +2(n-1)(U^2-U^4) =0.
  \]
  This allows $c_2>0$ to be arbitrary but forces
  \[
  c_4=\frac{4-n}{n-1}c_2^2\qquad\text{and}\qquad c_6=\frac
  {(n-4)(n-7)}{(n-1)^2}c_2^3.
  \]
  The claimed expansion follows.
\end{proof}

In function $\bry $ which we have used extensively in this paper is by
definition
\begin{equation}
  \label{eq:beta-defined}\bry (\sigma) = U(k\sigma)^2
\end{equation}
where $U$ is as above in Lemma~\ref{lem:BabyBryant}, and where the constant
$k$ is chosen so that
\begin{equation}
  \label{eq:beta-normalization}\bry (\sigma) = \frac1{\sigma^2} +
  o(\sigma^{-2}), \qquad(\sigma\to\infty).
\end{equation}
Lemma~\ref{lem:BabyBryant} implies that such a choice of $k$ exists, and
that $\bry $ is thus uniquely defined.

\end{document}